\documentclass[12pt,a4paper]{amsart}
\usepackage{amsmath}
\usepackage{amssymb}
\numberwithin{equation}{section}
 \usepackage{pdfsync}
\usepackage{ifthen}
\usepackage[a4paper]{geometry}
\usepackage{graphicx}
\newcommand{\OU}{\mathrm{OU}}
\newcommand{\bord}{\mathrm{bord}}
\newcommand{\DOU}{\mathrm{Dys}}
\newcommand{\finespace}{\mspace{1.0mu}}
\DeclareMathOperator{\Tr}{Tr}
\DeclareMathOperator{\vol}{vol}
\DeclareMathOperator{\Pf}{Pfaff}
\DeclareMathOperator{\spec}{spec}
\DeclareMathOperator{\diag}{diag}
\DeclareMathOperator{\GUE}{GUE}
\DeclareMathOperator{\GOE}{GOE}
\DeclareMathOperator{\GSE}{GSE}
\DeclareMathOperator{\minor}{minor}
\renewcommand{\Re}{\Realpart}
\DeclareMathOperator{\Realpart}{Re}
\title{Consecutive minors for Dyson's Brownian motions}


\author{Mark Adler}
\email{adler@brandeis.edu}
\address{%
Department of Mathematics, Brandeis University,
Waltham, Mass 02454, USA.}
\thanks{Adler: The support of a National Science Foundation grant 
\#~DMS-07-04271  is gratefully acknowledged.}
\author{Eric Nordenstam}
\email{eric.nordenstam@uclouvain.be}
\address{Department of Mathematics, Universit\'e de Louvain, 1348
Louvain-la-Neuve, Belgium.}
\thanks{Nordenstam: Supported by the 
``Interuniversity Attraction Pole" at UCL (Center of excellence): 
Nonlinear systems, stochastic processes and statistical mechanics (NOSY). }
\author{Pierre van Moerbeke} 
\email{pierre.vanmoerbeke@uclouvain.be \and
vanmoerbeke@brandeis.edu}
\address{Department of Mathematics, Universit\'e de Louvain, 1348
Louvain-la-Neuve, Belgium and Brandeis University,
Waltham, Mass 02454, USA.}
\thanks{
Van Moerbeke: 
The support of a National
Science Foundation grant \# DMS-04-06287, a European
Science Foundation grant (MISGAM), a Marie Curie Grant
(ENIGMA), FNRS and ``Inter-University Attraction Pole
(Belgium)" (NOSY) grants is gratefully acknowledged. 
 }

\newcommand{\Prob}[1][]{\mathbb{P}\ifthenelse{\equal{#1}{}}{}{\finespace[\finespace#1\finespace]}}
\newcommand{\Probs}[1][]{\mathbb{P}^*\ifthenelse{\equal{#1}{}}{}{\mspace{-1mu}[\,#1\,]}}

\newcommand{\UR}{\mathcal{U}}
\newcommand{\AR}{\mathcal{A}}
\newcommand{\HR}{\mathcal{H}}

\newcommand{\BC}{{\mathbb C}}

\newcommand{\BH}{{\mathbb H}}

\newcommand{\HB}{\HR^{(\beta)}_n}
\newcommand{\rg}{\rightarrow}
 \newcommand{\al}{\alpha}

\newcommand{\no}{\nonumber}

\newcommand{\la}{\langle}
\newcommand{\ra}{\rangle}
\newcommand{\ga}{\gamma}

\newcommand{\dt}{\delta}
\newcommand{\Dt}{\Delta}
 \newcommand{\vr}{\varepsilon}
\newcommand{\sg}{\sigma}
\newcommand{\BR}{{\mathbb R}}
\newcommand{\lb}{\lambda}

\newcommand{\dis}{\displaystyle}

\newcommand{\Oh}{{\mathcal O}}

\def\be#1\ee{\begin{equation}#1\end{equation}}
\def\bea#1\eea{\begin{eqnarray}#1\end{eqnarray}}
\def\bean#1\eean{\begin{eqnarray*}#1\end{eqnarray*}}

\theoremstyle{definition}

\theoremstyle{plain}
\newtheorem{theorem}[subsection]{Theorem}

\newtheorem{corollary}[subsection]{Corollary}

\theoremstyle{remark}
\newtheorem{remark}[subsection]{Remark}

\begin{document}
\nocite{Bor1,Manon2,J1,McK,Nord,Oksendal,OU}

\begin{abstract}
In 1962, Dyson \cite{Dyson} introduced dynamics in random matrix models, in particular into GUE (also for $\beta=1$ and $4$), by letting the entries evolve according to independent Ornstein-Uhlenbeck processes. Dyson shows the spectral points of the matrix evolve according to non-intersecting Brownian motions. The present paper shows that the interlacing spectra of two consecutive principal minors form a Markov process (diffusion) as well. 
This diffusion consists of two sets of Dyson non-intersecting Brownian motions, with a specific interaction respecting the interlacing. This is revealed in the form of the generator, the transition probability and the invariant measure, which are provided here; this is done in all cases: $\beta=1,~2,~4$. It is also shown that the spectra of three consecutive minors ceases to be Markovian for $\beta=2,~4$.


\end{abstract}

\maketitle

 \section{Introduction}

 In 1962, Dyson \cite{Dyson} introduced dynamics in random matrix models, in particular into GUE, by letting the entries evolve according to independent Ornstein-Uhlenbeck processes. According to Dyson, the spectral points of the matrix evolve according to non-intersecting Brownian motions. The present paper addresses the question whether taking two consecutive principal minors leads to a diffusion on the two interlacing spectra of the minors, taken together. This is so! The diffusion is given by the Dyson diffusion for each of the spectra, augmented with a strong coupling term, which is responsible for a very specific interaction between the two sets of spectral points, to be explained in this paper. However the motion induced on the spectra of three consecutive minors is non-Markovian, for generic initial conditions. A further question: is the motion of two interlacing spectra a determinantal process? We believe this is not the case; but determinantal processes appear upon looking at a different space-time directions. These issues are  addressed in another paper by the authors.
 
 During the last few years, the question of interlacing spectra for GUE-minors have come up in many different contexts. In a recent paper, Johansson and Nordenstam \cite{J-N}, based on domino tilings results of Johansson \cite{J1}, show that domino tilings of aztec diamonds provide a good discrete model for the consecutive eigenvalues of GUE-minors. In an effort to put some dynamics in the domino tiling model, Nordenstam \cite{Nord} then shows that the shuffling algorithm for domino tilings is a discrete version of an interlacing of two Dyson Brownian motions, introduced and investigated by Jon Warren \cite{Warren}. One might have suspected that the Warren process would coincide with the diffusion on the spectra of two consecutive principal minors. They are different!
 
Non-intersecting paths and interlaced processes (random walks and continuous 
processes) have been investigated by several authors in many different 
interesting directions; see e.g. \cite{Forr}, \cite{Jo02},\cite{J1}, 
\cite{J}, \cite{TrWi04}, \cite{Kat},\cite{OCon},  \cite{Warren1}, 
\cite{Manon1}, \cite{Kat}, 
 \cite{Koen},\cite{AvM05}, 
just to name a few. In particular, in ~\cite{TrWi04, AvM05}, partial 
differential equations were derived for the Dyson process and related processes.

 The plan of this paper is the following. We state precisely all the results 
in Section~\ref{sec:results}. 
Some usefull matrix equalities are derived in Section~\ref{sec:matrix} which 
are used in Section~\ref{sec:transition} to derive transition densities
for the various processes considered. 
Stochastic differential equations are derived in Sections~\ref{SDE1} 
and~\ref{SDE2}. 
The fact that the the spectra of three consecutive minors are not Markovian for generic initial conditions 
is demonstrated the last Section.

There is a companion paper by the same authors aiming at determining the kernel  for the point process related to the Dyson Brownian minor process along space-like paths \cite{ANV1}.

 \emph{Acknowledgement:} PvM thanks S.R.S. Varadhan for several insightful conversations
 in the beginning of this project.

\section{The Ornstein-Uhlenbeck process and 
Dyson's Brownian motion}
\label{sec:results}
 Consider the space $\HR^{(\beta)}_n$ of $n\times n$ matrices $  B$, with entries $ B_{k\ell} \in \BR$, $\BC$, $\BH$ ($\beta=1,~2,~4$) satisfying the symmetry conditions
 \be \label{B-matrix}
 B_{k\ell}=B_{\ell k}^\ast.
 \ee
Any element $ z  \in \BR$, $\BC$, $\BH$ admits a decomposition $z=z^{(0)}+\sum_{r=1}^{\beta-1} z^{(r)}e_r $, with $e_i$'s satisfying
$$e^2_1=e^2_2=e^2_3=-1 ,~  e_1e_2=-e_2e_1=e_3 ,~ e_1e_3=-e_3e_1=-e_2 ,~   e_2e_3=-e_2e_3=e_1 .$$ 
The conjugate ${}^\ast$ of an element  $ z \in \BR$, $\BC$, $\BH$ and its norm are given by 
\begin{align*}
   z^\ast &=z^{(0)}-\sum_{r=1}^{\beta-1} z^{(r)}e_r,&
   |z|^2=z  z^\ast=\sum_{r=0}^{\beta-1} z_{  }^{(r)2} ,
\end{align*}
  with $z$ admitting a polar decomposition 
  $ 
  z=|z| u, \text{ with }  |u |^2=
   \sum_{r=0}^{\beta-1} u^{(r)2}=1.
   $ 
The matrices $B\in \HR^{(\beta)}_n$, as in (\ref{B-matrix}), correspond to:

$$\HB=\left\{ \begin{aligned}
& \mbox{real symmetric $n\times n$ matrices, for $\beta =1$}\\
& \mbox{complex Hermitian  $n\times n$ matrices, for $\beta =2$}\\
 & \mbox{self-dual Hermitian $n\times n$ ``quaternionic" matrices, 
        for $\beta =4$}\\
\end{aligned}\right.
$$
with the compact groups of $\vol(\mathcal{U}_n^{(\beta)})=1$,
  \begin{equation}\label{Unitar}
\UR_n^{(\beta)}:=
\begin{cases}
O(n),&\beta =1\\
\UR(n),&\beta =2\\
{\rm Symp}(n),&\beta =4,
\end{cases}
\end{equation}
 acting on it by conjugation.
%

For $\beta=4$, it is well known that the quaternionic entries $z$ can be represented as follows
\be \begin{aligned}
\label{quat}
z=z^{(0)}+\sum_{r=1}^{\beta-1} z^{(r)}e_r & \longmapsto \hat z=\begin{pmatrix}  
                                         z^{(0)}+iz^{(1)}& z^{(2)}+iz^{(3)}\\
                                         -z^{(2)}+iz^{(3)}& z^{(0)}-iz^{(1)}
                                                           \end{pmatrix}.
 \end{aligned}\ee
 %
%
So, the $n\times n$ quaternionic matrices $B$ can be turned into $2n\times 2n$ self-dual Hermitian matrices $\hat B$, of which the real spectrum is doubly degenerate. Here, we shall define the $n$ distinct eigenvalues as the spectrum of $  B$. Unless stated otherwise we shall be working with the $n\times n$ quaternionic matrices, rather than the $2n\times 2n$ Hermitian matrices. Also, when working with matrices having quaternionic entries, the {\em trace} will be defined in the usual way, that is as the sum of the diagonal entries of the $n\times n$-matrix.

The {\em determinant} of a matrix $
 B\in \HR_n^{(4)}$ is given in terms of the $2n \times 2n$ matrix $\widehat B$ (as defined in (\ref{quat})), by the following procedure: first define the skew-symmetric $2n \times 2n$ matrix ${\mathbb B}$ by the following product:
 \begin{equation}
 \mathbb{B}:=\widehat B
 \cdot \left[\begin{pmatrix}0 & 1 \\-1 & 0\end{pmatrix} \otimes I_n\right],
 %
\label{bold-b} 
 \end{equation}
 and then ``$ \det B$" is defined as 
 \be
 \det B:=  \Pf ({\mathbb B})=(\det({\mathbb B}))^{1/2}=   \sum_p (-1)^{n-\ell} \prod^\ell_1 B_{\al \beta} B_{\beta \gamma} \ldots B_{\delta \al},
 \label{2.24'}
 \ee
 where $p$ is any permutation of the indices $(1,2,\ldots,n)$ consisting ot $\ell$ exclusive cycles of the form $(\al\to \beta \to \gamma \to \cdots \to \alpha)$; see Mehta \cite{Mehta}. In particular, this means that
 \begin{equation}
\det (\lb I -   B) = \prod^n_1 (\lb-\lb_i), \quad   \text{spec}~   B = \{\lb_1,\ldots,\lb_p\},
\label{quatdet}\end{equation}
with the $\lb_i$ being the double eigenvalues of ${\mathbb B}$.

The following normalization constant $Z^{-1}_{n,\beta}$  will come back over and over again:
\begin{equation} \label{N}
 Z^{-1}_{n,\beta} := 2^{-\frac n2}\left(\frac{\beta}{\pi}\right)^{N_{n,\beta}}, \hspace*{1cm}~\mbox{with  }
N:=N_{n,\beta}:=\frac{n}{2}+\frac{\beta}4 n(n-1)  .
\end{equation}

  \bigbreak
 
 Dyson's idea was to let the free parameters of the matrix evolve according to the SDE (Dyson process) 
\begin{equation}
 \label{1.3}
\begin{aligned}
dB_{ii}&=-B_{ii}dt+\sqrt{\frac{2}{\beta}}db_{ii}, & \quad& \text{$i=1$, \dots, $n$}\\
dB_{ij}^{(\ell)}&=-B_{ij}^{(\ell)}dt +\frac{1}{\sqrt{\beta}}db_{ij}^{(\ell)}, & \quad &\text{$1\leq i < j\leq n$ and $\ell = 0$, \dots, $\beta-1$,}\\
\end{aligned}
\end{equation}
where $d b_{ii}$, for $i=1$, \dots, $n$, 
 and $db_{ij}^{(\ell)}$, for $1\leq i<j\leq n$ and 
$l = 0$, \dots, $\beta -1$, are independent, standard Brownian motions. 
 Since the Ornstein-Uhlenbeck diffusions are independent, the Dyson process on the matrix $ B$ has a generator, which is just the sum of the OU-processes above:
\begin{equation}
 \label{DOUgen}
\mathcal{A} _{\DOU}:=\sum_{i = 1} ^n \left(\frac1\beta \frac{\partial^2}{\partial B_{ii}^2}-B_{ii}
 \frac{\partial}{\partial B_{ii}}\right)+
  \sum_{1\leq i<j\leq n} \sum_{\ell = 0}^{ \beta-1}
  \left(\frac1{2\beta} \frac{\partial^2}{\partial  B^{(\ell)2}_{ij}  }-B^{(\ell)}_{ij}
  \frac{\partial}{\partial B^{(\ell)}_{ij}}
 \right),
\end{equation}
 with transition probability
, setting $c:=e^{-t}$ and using the constant (\ref{N}),
\be\begin{aligned} \label{OU-tr}
 \Prob[ B_t\in d  B\,|\,   B_0=\bar B]&=:p(t,\bar B,  B)\,d  B\\
&=
 \frac{ Z^{-1}_{n,\beta}}{(1-c^2)^{N_{n,\beta}}}
  e^{-\frac{\beta}{2(1-c^2)}\Tr(  B-c\bar  B)^2}\,d  B,
\end{aligned}\ee
where $d  B$ is the product measure over all the independent parameters $B_{ii}$, $B_{ij}^{(\ell)}$. 
The transition probability (\ref{OU-tr}) satisfies the Fokker-Planck equation
\begin{equation} \frac{\partial p_{_{ }}}{\partial t}=\mathcal{A}_\DOU ^\top p_{_{ }} ,
\end{equation}
with 
\begin{equation} 
 \label{FP1}
\mathcal A_\DOU ^\top
=
 \frac{2}{\beta}\left(\frac{1}{2}\sum^n_{i=1} \frac{\partial}{\partial B_{ii}}  h^\beta\frac{\partial}{\partial B_{ii}} \frac{1}{h^\beta}+\frac{1}{4}   \sum_{1\leq i<j\leq n}\sum^{\beta -1}_{\ell =0} \frac{\partial}{\partial B_{ij}^{(\ell)}}  h^\beta  \frac{\partial}{\partial B_{ij}^{(\ell)}} \frac{1}{h^\beta}  \right),  
 \end{equation} 
with a delta-function initial condition, 
$ p_{ }(t,\bar B, B)\big\vert_{t=0}=\dt(\bar  B, B) 
,$  
and with invariant measure (density)  
\begin{equation}\label{Inv1}
\lim_{t\rg\infty} p_{ }(t,\overline{  B},  B)= Z_{n,\beta}^{-1}(h(  B))^\beta,\qquad
\text{with $h:=h(  B):=e^{-\frac{1}{2}{\rm Tr~}  B^2}$}.
\end{equation}
 Dyson discovered in \cite{Dyson} the surprising fact that the process restricted to 
$\spec(  B):=\{\lb_1,\lb_2,\dots, \lb_n\}$ is Markovian as well.
 This is the content of Dyson's celebrated Theorem (Theorem \ref{Th:Dyson}). 

Before stating the main Theorem, we define diagonal matrices $X=\diag (x_1,\ldots,x_n)$ and $Y=\diag (y_1,\ldots,y_n)$,~ vectors $w,v\in \BR^n,~\BC^n,~ {\mathbb H}^{n}$ and the inner-product $\langle w~,~v\rangle=\sum_1^n w_i  v_i^\ast$. Then consider the integral 
\be
 G _n^{( \beta)}(X,Y;w,v):=\int_{{\mathcal U}^{(\beta)}_n}
  dU e^{(\Tr XUYU^{-1}+2\Re \langle w,Uv\rangle)}
 ,\label{1.21}\ee
 and its integrand
 \be\label{Gint}
 {\mathcal G} _n^{ (\beta)}(U;X,Y;w,v):= e^{(\Tr XUYU^{-1}+2\Re \langle w,Uv\rangle)}
 .\ee
 %
%
For $w=v=0$, this is the more familiar integral
 \begin{equation}
 F_n ^{(\beta)}(X,Y):=G _n^{(\beta)}(X,Y;0,0)=\int_{\mathcal{U}_n^{(\beta)}}  dU
e^{ \Tr XU YU^{-1}}\,
,\end{equation}
which for $\beta=2$ gives the Harris-Chandra-Itzykson-Zuber formula:
\begin{equation}
F_n ^{(2)}(X,Y)=\frac{\det [e^{x_iy_j}]_{1\leq i,j\leq n}}{\Dt_n(x)\Dt_n(y)}\prod^n_{r=1} r!
.\label{HCIZ}\end{equation}
Does the integral (\ref{1.21}) admit such a representation ? {\em This is an open problem}.

\medbreak

%
%
%
In the following Theorem, formulae (\ref{1.4}), (\ref{1.5'}) and (\ref{1.28}) are due to Dyson \cite{Dyson}.

\begin{theorem}  \label{Th:Dyson}
 The Dyson process restricted to its spectrum $\spec(  B)=\lb:=\{\lb_1,\allowbreak
 \dots,\lb_n\} $ is Markovian with SDE given by:
%
%
\begin{equation} 
d\lb_i=\left(-\lb_i+\sum_{j\neq i}\frac{1}{\lb_i-\lb_j}\right)dt+\sqrt{\frac{2}{\beta}}
db_{ii},\quad \text{$i = 1$, \dots, $n$.}
\label{1.4}
\end{equation} 
Its transition probability\footnote{\label{footC}$C_{n,\beta}^{-1}$ is the norming constant for the Gaussian ensemble for general $\beta$, as obtained from the Selberg formula (see Mehta \cite{Mehta04}, formula (3.3.10)), (see (\ref{N}) for $N$) $$C_{n,\beta}^{-1}=(2\pi)^{-\frac{n}{2}}\beta^{ N} \prod^n_{j=1}\left(\frac{\Gamma \left(1+\frac{\beta  }{2}\right)}{\Gamma \left(1+\frac{\beta j}{2}\right)}\right).$$}, with $c:=e^{-t}$,   
\begin{multline}\label{1.20}
\Prob[\lb_t\in d\lb\bigr|  \lb_0=\bar\lb]=
p_{_{\lambda}}(t,\bar\lb,\lb)\,d\lb_1 \cdots d\lb_n\\
=\frac{C^{-1}_{n, \beta}}{(1-c^2)^{N_{n,\beta}}}e^{-\frac{\beta}{2(1-c^2)}\sum_1^n(\lb^2_i+c^2\bar\lb_i^{2})} F_n ^{(\beta)}\left(\frac{\beta c}{1-c^2}\lb,\bar\lb\right)|\Dt_n(\lb)|^{\beta}\prod_1^n d\lb_i
,\end{multline}
satisfies the Dyson diffusion 
equation, with delta-function initial condition 
($p_\lambda\big\vert_{t=0}=\dt(\lb,\bar\lb)$) (forward equation)
\begin{equation}
\frac{\partial p_{_{\lambda}}}{\partial t}=\AR^{\top}_{_{\lambda}} p_{_{\lambda}}, \text{ with }
\AR^{\top}_{_{\lambda}}:=\frac{1}{\beta}\sum^n_{i=1}\frac{\partial}{\partial\lb_i}\left(\Phi_n(\lb)\right)^{\beta}
\frac{\partial}{\partial \lb_i}\frac{1}{(\Phi_n(\lb))^{\beta}}.
\label{1.5'}
\end{equation}
The generator is
\begin{equation}
\mathcal{A}_\lambda= \sum_{i=1}^n\left(\frac1\beta  
\frac{\partial^2}{\partial \lb_i^2}+
  \left(-\lb_i  +\sum_{j\neq i}\frac{1}{\lb_i-\lb_j}\right)
  \frac{\partial }{\partial \lb_i }\right),
    \label{3.13''}
\end{equation}
and the invariant measure of the Dyson process on $  B$, projected onto 
$ \spec(B)$, is given by the $\GOE(n)$, $\GUE(n)$, $\GSE(n)$ measure for 
$\beta =1$, 2, 4 respectively:
\begin{equation}
C_{n,\beta}^{-1} \left(\Phi_n(\lb)\right)^{\beta}d\lb_1 \cdots  d\lb_n, \text{ with }
\Phi_n(\lb)=e^{-\frac{1}{2}\sum_1^n\lb^2_i}\vert\Dt_n(\lb)\vert. \label{1.28}
\end{equation}
\end{theorem}
For completeness we shall prove~(\ref{1.20}), (\ref{1.28}) in 
Section \ref{sec:transition} and~\eqref{1.4}, \eqref{1.5'} in 
Section~\ref{sec:ito}.

It is remarkable that the Dyson process is not only Markovian upon restriction to 
the spectrum of any {\em single} principal minor $B$, $B^{(n-1)}$, $B^{(n-2)}$, \ldots, 
of sizes $n$, $n-1$, $n-2$, \ldots, but 
also upon restriction to any {\em two consecutive} principal minors, in particular,
$$(\spec B, \spec  B^{(n-1)}):=(\lb ,\mu):=((\lb_1,\dots,\lb_n),(\mu_1,\dots,\mu_{n-1})),$$
with intertwining property
 \begin{equation}
\lb_1\leq\mu_1\leq\lb_2\leq\mu_2\leq \cdots \leq\lb_{n-1}\leq\mu_{n-1}\leq\lb_n.
\end{equation}%
We denote by $\AR_{\lb}$ and $\AR_{\mu}$ the generators of the consecutive spectra ${\rm spec~}  B$ and  ${\rm spec~}  B^{(n-1)}$, as defined in (\ref{1.5'}). %
Define the characteristic polynomials of the two consecutive minors $B$ and $B^{(n-1)}$,
\begin{align}\label{1.10}
P_n(x)&=\prod^n_{\al =1}(x-\lb_{\al}), & P_{n-1}(x)&=\prod^{n-1}_{\beta =1}(x-\mu_{\beta}),
\end{align}
and the Vandermonde determinants
\begin{equation}
  \label{vdm}
  \begin{split}
    \Dt_n(\lb)&:=\prod_{j>i}(\lb_j-\lb_i)\geq 0, \\
    \Dt_n(\lb,\mu)&:= \prod_{i=1}^n \prod_{j=1}^{n-1} (\lb_i-\mu_j)
    =\prod_1^n P_{n-1}(\lb_i)=\prod_1^{n-1} P_{n}(\mu_i),
  \end{split}
\end{equation} 
 with $\Dt_n(\lb,\mu)(-1)^{\frac{n(n-1)}{2}}\geq 0$ because of the intertwining.

 In order to state Theorem \ref{Th:Dyson2}, we need the following property of any matrix $B\in \HR^{(\beta)}_n $; not only can $B$ be conjugated by a matrix $U^{(n)}\in \UR_n^{(\beta)}$  (see (\ref{Unitar})), in the standard way, such that 
 \be\label{conj1}
 (U^{(n)})^{-1}BU^{(n)}=\mbox{diag}(\lb_1,\ldots,\lb_n),
 \ee
 but also by a matrix of the form $\begin{pmatrix}  U^{(n-1)}&0\\
                                     0&1\end{pmatrix} $, with $U^{(n-1)}\in \UR_{n-1}^{(\beta)}$, to yield a bordered matrix $B_\bord$ :
\begin{equation}\begin{aligned}
\begin{pmatrix}  U^{(n-1)}\!&\!0\\
                                     0\!&\!1\end{pmatrix} 
             B
             \begin{pmatrix}  U^{(n-1)}\!&\!0\\
                                     0\!&\!1\end{pmatrix} ^{-1}                        
  %
              =                                \begin{pmatrix}
  \mu_1 &        0    &   \cdots     &        0    &    r_1u_1 \\
     0          &\mu_2 &   \cdots     &       0    &     r_2u_2 \\
   \vdots            &   \vdots     & \ddots &   \vdots      &    \vdots  \\
  0            &        0 &  \cdots &  \mu_{n-1} &   r_{n-1}u_{n-1} \\            
  r_1  u^\ast_1&r_2  u^\ast_2& \cdots & r_{n-1}  u^\ast_{n-1} & r_n 
  \end{pmatrix}=:B_\bord
\end{aligned}\label{2.22}
\end{equation}
 with $| u_{i} | =1$ (angular variables) and with $r_i\geq 0$ for $1\leq i\leq n-1$ and $r_n$, given by
 \begin{align}\label{r}
r^2_k &:=-\frac{P_n(\mu_k)}{P'_{n-1}(\mu_k)} \geq 0, ~~1\leq k\leq n-1,&  r_n&:=\sum_1^n\lb_i-\sum_1^{n-1}\mu_i
.\end{align}
The conjugation in (\ref{2.22}) 
 transforms the last column $v$ of $B$ in the last column of the bordered matrix $B_\bord$ (except for the last entry); i.e.,
\be\label{v}
U^{(n-1)} v=(r_1u_1,\ldots,r_{n-1}u_{n-1})^\top, ~\mbox{and} ~B_{nn}=r_n,~\mbox{with } ~v:=(B_{1,n},\ldots,B_{n-1,n})^\top.
\ee
 These facts, (\ref{2.22}), (\ref{r}) and (\ref{v}), will be discussed and shown in Section \ref{sec:matrix}.

 \bigbreak

The next statement is the analogue of Theorem \ref{Th:Dyson} for the case of the spectra of two consecutive minors.

\begin{theorem}\label{Th:Dyson2}The Dyson process on $ B$ restricted to 
 $$(\spec  B,\spec  B^{(n-1)})=(\lb,\mu):=((\lb_1,\dots,\lb_n),(\mu_1,\dots,\mu_{n-1}))$$
 is a diffusion $(\lb(t),\mu(t))$ as well, with the following SDE: 
\begin{equation}
\label{SDE-lambda0}
\begin{split}
  d\lb_{\al}=&
\left(-\lb_{\al}+\sum_{\vr\neq\al}\frac{1}{\lb_{\al}-\lb_{\vr}}\right)dt  
+\sqrt{\frac{2}{\beta}}
\frac{P_{n-1}(\lb_{\al})}{P'_n(\lb_{\al})
} \\
&  \times {\left(\sum_{1\leq i<j\leq n-1}\!\frac{\sqrt{2}~r_ir_j~\tilde{d  b_{ij}}}{(\lb_{\al}\!-\!\mu_i)(\lb_{\al}\!-\!\mu_j)}+\sum^{n-1}_{i=1}\frac{r^2_idb_{ii}}{(\lb_{\al}\!-\!\mu_i)^2}\right.}
 \displaystyle{\left.+\sum_{i=1}^{n-1}\frac{\sqrt{2}~r_i \tilde{d  b_{in}}}{\lb_{\al}\!-\!\mu_i}+db_{nn}\!\right),
  }  \\
  d\mu_{\ga}=&\left(-\mu_{\ga}+\sum_{\vr\neq\ga}\frac{1}{\mu_{\ga}-\mu_{\vr}}\right)dt +\sqrt{\frac{2}{\beta}}db_{\ga\ga},
\end{split}
\end{equation}
in terms of independent standard Brownian motions $\{db_{ii},\tilde{d   b_{ij}}\}_{1\leq i<j\leq n}$. 
 Its transition probability\footnote{\label{footZhat}The constant reads  $$
\hat Z^{-1}_{n,\beta}=\displaystyle{\frac{\beta^{N_{n,\beta}} (\Gamma(1+\frac{\beta}{2}))^{n-1}}{(2\pi)^{\tfrac n2}\pi^{\frac \beta 2 (n-1)}\prod^{n-1}_{j=1}\Gamma (1+\frac{\beta j}2)
(\vol(S^{\beta-1}))^{n-1}  }},
$$
 and 
$\vol (S^k)=2\pi \prod_{i=1}^{k-1} \left(2\int_0^{\pi/2}(\cos \theta)^{i }d\theta\right)$ for $k\geq 2$,
$\vol(S^0)=1$ and  $\vol(S^1)=2\pi$,  which is proved by induction on $k$; so $\vol(S^2)=4\pi $ and $\vol(S^3)=2\pi^2$.} is given by:
%
%
\be\begin{aligned}
\label{1.24}
\lefteqn{\hspace{-1.5cm}p_{\lambda\mu} \left(t,(\bar\lb,\bar\mu),(\lb,\mu)\right)\,d\lb\,d\mu 
}\\
 =& \Prob [(\lb_t,\mu_t)\in \, (d\lb,d\mu) \, |\, (\lb_0,\mu_0)=(\bar\lb,\bar\mu)]
 \\
 =&\frac{\hat Z^{-1}_{n,\beta}}{(1-c^2)^{N}}e^{-\frac{\beta}{2(1-c^2)}\sum_1^n(\lb^2_i+c^2\bar\lb_i^{2})}
 \int_{(S^{\beta-1})^{2(n-1)}}\prod_1^{n-1}d\Omega^{(\beta-1)}(u_i)
d\Omega^{(\beta-1)}(\bar u_i) 
\\
& \times 
G _{n-1}^{( \beta)} \Bigl(\frac{\beta c}{1-c^2} \mu,\bar \mu; ~\frac{\beta c}{1-c^2}(r_iu_i)_{1}^{n-1},
(\bar r_i \bar u_i)_{1}^{n-1}\Bigr) \\
&\times  e^{\frac{\beta c r_n\bar r_n}{1-c^2}}  
 |\Dt_n(\lb)\Dt_{n-1}(\mu)|~|\Dt_n(\lb,\mu)|^{(\frac{\beta}{2}-1)} \prod^{n}_{1}d\lb_i\prod_1^{n-1}d\mu_j,
\end{aligned}\ee
where the $r_i$'s are given by (\ref{r}). It is also a solution of the following forward diffusion equation, with delta-function initial condition 
\begin{equation}
\frac{\partial p_{\lambda\mu}}{\partial t}=\AR
 ^\top p_{\lambda\mu},
\text{ with }
\AR^\top
 := \AR_{\lb}^\top  +\AR_{\mu}^\top +\AR^\top_{\lb\mu} ,
\label{1.12'}\end{equation}
 where
\begin{equation}
\label{1.12''}
\AR^\top_{\lb \mu}:=
-\frac{2}{\beta}
\sum_{i = 1}^n  \sum_{j=1}^{n-1}
\frac{\partial}{\partial \lb_i} \frac{\partial}{\partial \mu_j}
  \left( \frac{1}{(\lb_i-\mu_j)^2}\frac{P_{n-1}(\lb_i)}{P'_n(\lb_i)}
  \frac{P_n(\mu_j)}{P'_{n-1}(\mu_j)}\right)
\end{equation}
and where $\mathcal{A}_\lambda^\top $ and
$\mathcal{A}_\mu^\top $ are defined by~\eqref{1.5'}.
The Dyson process restricted to $(\lb,\mu)$ has invariant measure, (see (\ref{vdm})),
\begin{equation} 
  \hat Z_{n,\beta}^{-1} \left( {\vol (S^{\beta-1})}{ }\right)^{2(n-1)}
e^{-\frac{\beta}{2}\sum_1^n\lb^2_i}|\Dt_n(\lb)\Dt_{n-1}(\mu)|~|\Dt_n(\lb,\mu)|^{ \frac{\beta}{2}-1 }\prod^n_1d\lb_i\prod_1^{n-1}d\mu_i
.\label{1.29}\end{equation} 
\end{theorem}
The SDE~\eqref{SDE-lambda0} and generator~\eqref{1.12'} are computed in 
Section~\ref{SDE2} while the espressions for transition density~\eqref{1.24} 
and invariant measure~\eqref{1.29} are proved in Section~\ref{sec:transition}.
 
Note that it is an immediate consequence of Theorem~\ref{Th:Dyson} 
that the generator $\mathcal{A}_\DOU$, defined in
 (\ref{DOUgen}), acting on the $\lb_i$ and $\mu_i$, has the form
\begin{equation}
\mathcal{A}_\DOU (\lb_i)=\mathcal{ A}_\lambda(\lb_i)
\quad\text{ and } \quad
\mathcal{A}_\DOU (\mu_i)=\mathcal{ A}_\mu(\mu_i)
,\label{restr}
\end{equation}
where $\mathcal{ A}_\lambda$ and $\mathcal{ A}_\mu$ are
defined by~\eqref{3.13''}.

\bigbreak

Whereas all statements in this paper hold for $\beta=1,2,4$, a part of it can be extended to general $\beta>0$, as will be shown in section \ref{SDE2}, after the proof of Theorem \ref{Th:Dyson2}:

\begin{corollary}\label{general}

For general $\beta>0$, the SDE  (\ref{SDE-lambda0}), in terms of the independent standard Brownian motions $\{db_{ii},\tilde{d   b_{ij}}\}_{1\leq i<j\leq n}$, defines a diffusion, whose generator is given by the same equations (\ref{1.12'}), and whose invariant measure is given by (\ref{1.29}). Moreover, this diffusion restricted to the $\lambda_i$'s (or to the $\mu_i$'s) is the standard Dyson Brownian motion (\ref{1.4}). 

\end{corollary}

The following corollary shows that the $\mu_i$'s in $\lb_{i}\leq\mu_i\leq \lb_{i+1}$ are repelled by the boundary and {\em fluctuate in unison} with the boundary points, when they get close.  

\begin{corollary}\label{Cor:unis}
The nonnegative gaps $\mu_i-\lb_i$ and $\lb_{i+1}-\mu_i$ for $1\leq i\leq n-1$ satisfy, in the notation of (\ref{SDE-lambda0}), 

\be\begin{aligned}\label{unis}
d(\mu_i-\lb_i)&=F_i(\lb,\mu)dt+\sqrt{\mu_i-\lb_i}\sum_{1\leq k\leq \ell\leq n} \alpha_{k\ell}d\tilde b_{k \ell}
\\
d(\lb_{i+1}-\mu_i)&=\hat F_i(\lb,\mu)dt+\sqrt{\lb_{i+1}-\mu_i}\sum_{1\leq k\leq \ell\leq n} \hat\alpha_{k \ell}d\tilde b_{k \ell}
\end{aligned}\ee
with
\be\left\{\begin{aligned}
& \mbox{some}~~\alpha_{ k  \ell }= { \Oh}(1) ~~\mbox{for} ~~\mu_i\simeq\lb_i
~~~ \mbox{and some}~~~
\hat \alpha_{ij}={ \Oh}(1) ~~\mbox{for} ~~\mu_i\simeq\lb_{i+1} .
\\
&  F_i(\lb,\mu)\bigr|_{ \mu_i=\lb_i}>0, ~~~~~~~~
\hat F_i(\lb,\mu)\bigr|_{\mu_i=\lb_{i+1}}>0
 \end{aligned}\right.\ee

\end{corollary}

This is to be compared with the  Warren process \cite{Warren}, 
which also describes two intertwined Dyson processes $\lb$ and $\mu$, 
but with an entirely different interaction: namely the $\mu_i$'s 
near the boundaries of the intervals $[\lb_i,\lb_{i+1}]$ behave like 
the absolute value of one-dimensional Brownian motion near the origin.


\newcommand{\tb}{\tilde B}


 \bigbreak

As we saw, the Dyson process on $  B$, restricted to the spectrum of one 
principal minor or the spectra of two consecutive minors leads to two 
Markov processes; opposed to that, we have the following statement, which will be proved in Section~\ref{sec:threeminors}.
\begin{theorem}\label{Theo:3minors}
The restriction of the Dyson process restricted to the following data 
 $$(\spec  B,\spec B^{(n-1)},\spec B^{(n-2)}):=(\lb,\mu,\nu) $$ is \emph{not} 
Markovian for generic initial conditions on $B$, i.e., the joint spectra of any three neighbouring set of minors 
of $  B$ are not Markovian, for $\beta=2$ and $4$. 
\end{theorem}

\section{Some Matrix Identities}
\label{sec:matrix}

In this section, we prove formula (\ref{2.22}) for $r_k$ as in (\ref{r}). In the course of doing that we will also prove the following formulas:
\begin{equation} 
\sum_1^{n-1}r^2_i+\frac{r^2_n}{2}=\frac{1}{2}\left(\sum_1^n\lb_i^2-\sum_1^{n-1}\mu^2_i\right)
\quad\text{  and  } \quad
 \prod_1^{n-1}r^2_i=
  \frac{|\Dt_n(\lb,\mu)|}{\Dt^2_{n-1}(\mu)}.\label{2.16}\end{equation} 
One also has the (often used) identities
\begin{equation}
  \sum_{i=1}^{n-1}\frac{r_i^2}{\lb_\ell-\mu_i}+r_n -\lb_\ell=0 
\quad \text{ and }\quad
\frac{P_n'(\lb_\ell)}{P_{n-1}(\lb_\ell)}=\sum_{i=1}^{n-1}\left(\frac{r_i}{\lb_\ell-\mu_i}\right)^2+1
\label{2.28}
.\end{equation}
 Finally, one has, for fixed $ (\mu_1, \dots,\mu_{n-1})$ and fixed $(u_1,\dots,u_{n-1})$, 
%
\begin{equation} 
\prod_1^{n-1}dr^2_j~dr_n =(-1)^{n-1}\frac{\Dt_n(\lb)}{\Dt_{n-1}(\mu)}\prod_1^{n }d\lb_i
\label{2.15}.\end{equation} 

\begin{proof}
From the form of the matrix $B_\bord$ as in (\ref{2.22}), one checks (see (\ref{vdm}) and also the formula (\ref{2.24'}) for the determinant in the quaternionic case)
\begin{equation} \prod_1^{n } (\lb_i-\lb) = \det (B_\bord-\lb I)=
 \prod_1^{n-1} (\mu_i-\lb) 
 \left( \sum_{i=1}^{n-1}\frac{r_i^2}{\lb-\mu_i}+r_n-\lb\right)
 ,\label{2.24}\end{equation} 
 from which it follows that \footnote{the $\sg_k(\lb)$ are symmetric polynomials: $\sg_1(\lb)=\sum_i\lb_i$,   $\sg_2(\lb)=\sum_{i<j}\lb_i\lb_j$,  etc\dots.
The same for $\sg_k(\mu)$}
  \begin{multline}
 - \sum_{i=1}^{n-1}\frac{r_i^2}{\lb-\mu_i}-r_n +\lb
= \frac{P_n(\lb)}{P_{n-1}(\lb)}\\
  =\lb-(\sg_1(\lb)- \sg_1(\mu))
 -\frac1\lb \left( \sg_1(\lb)\sg_1(\mu)+\sg_2(\mu)-\sg_2(\lb)-\sg_1^2(\mu)\right)+O(\frac1{\lb^2}).
 \label{2.25}\end{multline} 
Then taking residues in formula (\ref{2.25}) yields the first 
formulae~(\ref{r}) and thus the formula for $\prod_1^{n-1} r_i^2$ 
in~(\ref{2.16}). Comparing the coefficients of  $\lb^0$ and the $\lb^{-1}$ 
on both sides of~(\ref{2.25}) yields the first formula of~(\ref{2.16}).  
 Setting $\lb=\lb_ \ell$ in the expression (\ref{2.25}) and its derivative 
with regard to $\lb$ implies the two sets of $n$ identities (\ref{2.28}), 
in view of the definition of $P_n$.
Formula (\ref{2.15}) amounts to computing the Jacobian determinant of the 
transformation from~$ \lb_1$, \ldots, $\lb_n$ to~$r_1$, \ldots, $r_n$; 
to do this, take the differential 
 of the first of the $n$ expressions appearing 
in (\ref{2.28}) (as functions of~$ \lb_1$, \ldots, $\lb_n$ and~$r_1$, 
\ldots, $r_n$), keeping the $\mu_i$'s fixed and use the second of the 
expressions (\ref{2.28}):
\begin{equation}
\begin{aligned}
 0&= \sum_{i=1}^{n-1} \frac{dr_i^2}{\lb_\ell-\mu_i}+dr_n-\Bigl(1+\sum_{i=1}^{n-1}
 \frac{ r_i^2}{(\lb_\ell-\mu_i)^2}\Bigr)d\lb_\ell\\
 &=  \sum_{i=1}^{n-1} \frac{dr_i^2}{\lb_\ell-\mu_i}+dr_n-
    \frac{P_n'(\lb)}{P_{n-1}(\lb)} d\lb_\ell
 ,   \end{aligned}
\label{2.27}\end{equation}
  which in matrix form reads
    $$
  \Gamma
   \begin{pmatrix}
    dr_1^2\\
    dr_2^2\\
    \vdots  \\
    dr_{n-1}^2\\
    dr_n
     \end{pmatrix}
  =\diag\left( \frac{P_n'(\lb_1)}{P_{n-1}(\lb_1)}, 
    \ldots, \frac{P_n'(\lb_n)}{P_{n-1}(\lb_n)}\right)
  \begin{pmatrix}
    d\lb_1 \\
    d\lb_2 \\
        \vdots  \\
    d\lb_{n-1} \\
         d\lb_n
     \end{pmatrix}
     ,$$
where (by Cauchy's determinantal formula)
\begin{equation} \Gamma:=
 \begin{pmatrix}
  \frac1{\lb_1-\mu_1}&\frac1{\lb_1-\mu_2}&\ldots& \frac1{\lb_1-\mu_{n-1}} &1\\
   \frac1{\lb_2-\mu_1}&\frac1{\lb_2-\mu_2}&\ldots& \frac1{\lb_2-\mu_{n-1}} &1\\
   \vdots&\vdots&                                                  &   \vdots&          \vdots\\
    \frac1{\lb_n-\mu_1}&\frac1{\lb_n-\mu_2}&\ldots& \frac1{\lb_n-\mu_{n-1}} &1\\
   \end{pmatrix},
\text{with }\det \Gamma=(-1)^{(n-1)(\frac n2 +1)} \frac{\Dt_n(\lb)\Dt_{n-1}(\mu)}{\Dt_n(\lb,\mu)}
   .\label{2.29}\end{equation}
   The formula (\ref{2.29})  for the determinant follows from the observation that $\det \Gamma$ has 
homogeneous degree $1-n$ and vanishes when $\Dt_n(\lb)\Dt_{n-1}(\mu)$ does and blows up (simply) when 
and only when $\Dt_n(\lb,\mu)$ vanishes. Thus we have
    $$
   \frac{\partial(r_1^2,\ldots,r_{n-1}^2,r_n)}{\partial (\lb_1,\ldots,\lb_n)}= 
\prod_1^{n}  \frac{P_n'(\lb_i)}{P_{n-1}(\lb_i)}(\det \Gamma)^{-1}
 =(-1)^{n-1}\frac{\Dt_n(\lb)}{\Dt_{n-1}(\mu)} . $$
This concludes the proof of formulas stated in the beginning of this section.
\end{proof}


\section{Transition Probabilities}
\label{sec:transition}

A quick review of the Ornstein-Uhlenbeck process (see Feller \cite{Feller}): it is a diffusion on $\BR$, given by the one-dimensional SDE,
\begin{equation}
dx=-\rho x\, dt+\frac{1}{\sqrt{\beta}}\, db
,\label{T1}\end{equation}
and it has transition probability $(c:=e^{-\rho t})$
\begin{align*}
\Prob[x_t\in dx\,\vert\, x_0=\bar x]&=:
p_\OU(t; \bar x,x)\,dx\\
\\
&=  { \Bigl(\frac{ \rho\beta}{\pi (1-c^2)}\Bigr)^{1/2}}
\exp\left(-\frac{\rho\beta(x-c\bar x)^2}{ 1-c^2 }\right)\,dx.
\end{align*}
The transition probability is a solution of the forward (diffusion) equation, with $\delta$-function initial condition~\footnote{The backward equation becomes the heat equation with $(x,t)\mapsto (xe^{\rho t},\frac{1-e^{2\rho t}}{2\rho})$} 
\begin{equation}
\frac{\partial p_\OU}{\partial t}=
\left(\frac{1}{2\beta}\frac{\partial^2}{\partial x^2}- 
\frac{\partial}{\partial x}(-\rho x)\right)p_\OU
=\frac{1}{2\beta}\left(\frac{\partial}{\partial x}
\phi_{\rho\beta}(x)\frac{\partial}{\partial x}
\frac{1}{\phi_{\rho\beta}( x)}\right)p_\OU,
\label{T3}\end{equation}
and invariant measure (density)
\begin{equation}
\label{T4}
\phi_{\rho\beta}(x)=\sqrt{\frac{\rho\beta}{\pi}} e^{-\rho\beta x^2}=
\lim_{t\rg\infty}p_\OU(t;\bar x,x).
\end{equation}
\begin{proof}[Proof of transition probabilities (\ref{OU-tr}), (\ref{1.20}) and (\ref{1.24}) ]  \hspace*{2cm} 


\noindent {\bf (i)} {\em The Fokker-Planck equation for the transition probability of the Dyson process. }The Dyson process consists of running the free parameters of the matrix $B\in \HB$, as in (\ref{B-matrix}), according to independent Ornstein-Uhlenbeck processes, with $\rho=1$, the diagonal with $\beta \rg \beta/ 2$ and the off-diagonal parameters with $\beta \rg \beta $. Remembering the definition (\ref{N}) of $N=N_{n,\beta}$ and the definition of the trace (after (\ref{quat})), one has, setting $ c=e^{-t} $, and using (\ref{1.3}), (\ref{T1}) and (\ref{T3}), the transition probability for the Dyson process is given by\footnote{with constant $Z^{-1}_{n,\beta}=2^{-n/2}(\frac \pi \beta)^ {-N_{n,\beta}}$.}
\begin{align}
p (t,\bar B,B)=&\prod^n_{i=1}
p_{\OU}(t;\bar B_{ii},B_{ii})\prod_{1\leq i<j\leq n}
\prod_{\ell=0}^{\beta -1}
p_{\OU}(t;\bar B_{ij}^{(\ell)}, B_{ij}^{(\ell)})\label{T5} \\
=&\prod^n_{i=1}\left(\frac{e^{-\beta\frac{(B_{ii}-c\bar B_{ii})^2}{2(1-c^2)}}}{ ({2\pi}(1-c^2)/\beta)^{\frac{1}{2}}}\right)
\prod_{1\leq i<j\leq n}\prod_{\ell = 0 } ^{\beta -1} 
 \left(\frac{e^{-\beta\frac{(B_{ij}^{(\ell)}
-c\bar B_{ij}^{(\ell)})^2}{(1-c^2)}}}{ ( \pi (1-c^2)/ \beta)^{\frac{1}{2}}}\right) 
\nonumber\\
=&\frac{1}{2^{n/2}\bigl(\frac \pi \beta (1-c^2) \bigr)^{N_{n,\beta}}}e^{-\frac{\beta}{2(1-c^2)}\Tr(B -c\overline{B})^2}
= \frac{Z_{n,\beta}^{-1}}{(1-c^2)^{N_{n,\beta}}}e^{-\frac{\beta}{2(1-c^2)}\Tr(B -c\overline{B})^2},
\nonumber
\end{align}
 yielding (\ref{OU-tr}), while $\lim_{t\rg\infty} p_{ }(t,\overline{  B},  B)= Z_{n,\beta}^{-1}(h(  B))^\beta$ is immediate, showing (\ref{Inv1}).
  Moreover, from~(\ref{T3}), one computes for $p(t;\bar B,B)$,
\begin{align}
\frac{\partial }{\partial t}p(t;\bar B,B)=
&\,\frac{2}{\beta}\biggl[\frac{1}{2}\sum^n_{i=1}
\frac{\partial}{\partial B_{ii}}\phi_{\beta/2} (B_{ii}) \frac{\partial}{\partial B_{ii}}\frac{1}{\phi_{\beta/2} (B_{ii})} 
\\&
\qquad+\,\frac{1}{4}\sum_{1\leq i<j\leq n} \sum_{\ell = 0} ^{\beta -1}
\frac{\partial}{\partial B^{(\ell)}_{ij}}\phi_{\beta } (B^{(\ell)}_{ij}) \frac{\partial}{\partial B^{(\ell)}_{ij}}\frac{1}{\phi_{\beta } (B^{(\ell)}_{ij})}  
 \biggr]p(t;\bar B,B)
 \no\\
=&\,\frac{2}{\beta}\biggl[ \frac{1}{2}\sum^n_{i=1}
\frac{\partial}{\partial B_{ii}}h(B) \frac{\partial}{\partial B_{ii}}\frac{1}{h (B)}
\\
&\qquad +\frac{1}{4}\sum_{1\leq i<j\leq n} \sum_{\ell = 0} ^{\beta -1}
\frac{\partial}{\partial B^{(\ell)}_{ij}}h (B)\frac{\partial}{\partial B^{(\ell)}_{ij}} \frac{1}{h(B)}\biggr]p(t;\bar B,B)\no
\end{align}
with
\bea
h(B)=\text{constant}\times\prod^n_{i=1}\phi_{\beta /2}(B_{ii})\prod_{1\leq i<j\leq n\atop{0\leq\ell\leq\beta -1}}\phi_{\beta}(B_{ij}^{(\ell)}),
\eea
proving (\ref{FP1}).

\medbreak

\noindent {\bf (ii)}  {\em The transition probability (\ref{1.20}) and the invariant measure for the $\lambda_t$ process}. Set 
$$
\lb =(\lb_1,\ldots,\lb_n) ~~~
\text{ and } ~~~\bar\lb =(\bar\lb_1,\ldots,\bar\lb_n).
$$
By the Weyl integration formula, given $B = U \lb U^{-1}$ and initial 
condition $\bar{B}=\bar U\bar\lb \bar U^{-1}$, express $dB=d(U\lb U^{-1})$ 
in formula (\ref{OU-tr}) in terms of spectral and angular variables 
$dB=Z_{n,\beta} C^{-1}_{n,\beta} |\Dt_n(\lb)|^{\beta}dU\prod_1^nd\lb_i$, with Haar measure $dU$ on 
$ \UR_n^{(\beta)} $ normalized such that {\rm Vol}$(\UR_n^{(\beta)})=1$, with $C^{-1}_{n,\beta}$  
defined in footnote \ref{footC}. 
This yields, using the transition probability (\ref{T5}),  \begin{multline} 
\label{2.8} 
\Prob[B_t\in dB  \, \vert\,  B_0=\bar{B}]=p(t;\bar B,B) dB\\
\hspace*{.88cm}=\frac{C^{-1}_{n,\beta}}{(1-c^2)^{N}}e^{-\frac{\beta}{2(1-c^2)}\sum_1^n(\lb_i^2+c^2\bar\lb^2_i)}\\
\times ~e^{\frac{\beta c}{1-c^2} \Tr U \lb U^{-1} \bar U\bar\lb \bar U^{-1}}|\Dt_n(\lb)|^{\beta}dU\prod_1^nd\lb_i .
\end{multline}
Note that the constant $C^{-1}_{n,\beta}$ is compatible with the fact that for $t\to \infty$ this transition probability tends to the GUE-probability; see below.

We now compute the transition probability $p_{\lb }(t; \bar \lb,  \lb )d\lb$ for the spectrum of the Dyson process; this will be a model to compute the transition probability for the $(\lambda_t,\mu_t)$-process. So, consider
\be\begin{aligned}
\lefteqn{p_{\lb }(t; \bar \lb,  \lb )d\lb  
= \Prob_{ }(\lb_t\in d \lb \bigr| \lb_0=\bar{\lb})} 
\\
=&\int_{B _t\in \HB(  \lambda )}\int_{\bar B  \in \HB(\bar \lambda ) }
\Prob \left[B^{ }_t\in dB~\big\vert ~B^{ }_0=\bar B\right]
 \Prob   \left[B^{ }_0\in d\bar B \big\vert   \mbox{spec}(B_0^{ }) = \bar \lb   \right]
\\
\\
=&\int_{\left(\UR_ n^{(\beta)}\right) ^2}
\frac{C^{-1}_{n,\beta}}{(1-c^2)^{N}}e^{-\frac{\beta}{2(1-c^2)}\sum_1^n(\lb_i^2+c^2\bar\lb^2_i)} 
 e^{\frac{\beta c}{1-c^2} \Tr U \lb U^{-1} \bar U\bar\lb \bar U^{-1}}|\Dt_n(\lb)|^{\beta}dUd\bar U\prod_1^nd\lb_i 
\\
=&
 \frac{C^{-1}_{n, \beta}}{(1-c^2)^{N_{n,\beta}}}e^{-\frac{\beta}{2(1-c^2)}\sum_1^n(\lb^2_i+c^2\bar\lb_i^{2})} F_n ^{(\beta)}\left(\frac{\beta c}{1-c^2}\lb,\bar\lb\right)|\Dt_n(\lb)|^{\beta}\prod_1^n d\lb_i ,\end{aligned}\ee
 using (\ref{2.8}) above, using the following conditional probability formula:
 \be\begin{aligned}
\label{}
 {\Prob   \left[B^{ }_0\in d\bar B \big\vert  \mbox{spec} (B^{ }_0) =  \bar \lb  \right]} 
 =
\frac{\Prob   \left[B^{ }_0\in d\bar B ,~  \mbox{spec} (B^{ }_0) \in  d\bar \lb  \right]}
{\Prob   \left[  \mbox{spec} (B^{ }_0) \in d\bar \lb  \right]}=d\bar U^{ }
\end{aligned}\ee
and finally using the integration (\ref{1.21}),
%
\begin{multline*}
\int_{\UR_n^{(\beta)}}e^{\frac{\beta c}{1-c^2} \Tr U \lb U^{-1} \bar U\bar\lb \bar U^{-1}}dU 
=\int_{\UR_n^{(\beta)}}e^{\frac{\beta c}{1-c^2} \Tr \lb U \bar \lb U^{-1} }dU
=F_{n}^{(\beta)}\left(\frac{\beta c}{1-c^2}\lb,\bar\lb\right)
,\end{multline*}
thus yielding (\ref{1.20}). 

Letting $t\rg  \infty$, (equivalently $c\rightarrow 0$) 
in~\eqref{2.8} proves formula (\ref{1.28}) for the invariant measure, taking into account that $F_{n}^{( \beta)}(0,Y)={\rm vol}(\UR_n^{\beta})=1$.

\medbreak

\noindent {\bf (iii)}  {\em Proof of the transition probability (\ref{1.24}) and the invariant measure for the $(\lambda_t, \mu_t)$ process}.  
The proof of (\ref{1.24}) in Theorem \ref{Th:Dyson2} proceed along similar lines. First observe the identity 
 \be \begin{aligned}
 \lefteqn{\Prob   [B_t \in dB~ \big\vert ~(\mbox{spec} (B^{(n)}_0),\mbox{spec}(B^{(n-1)}_0))=(\bar \lb,\bar \mu) ]}\\
 &=
 \int_{\bar B\in \HB(\bar \lambda,\bar \mu)}\Prob \left[B_t\in dB~\big\vert ~B_0=\bar B\right]
~ \Prob   \left[B_0=d\bar B \big\vert (\mbox{spec} (B_0^{(n)} ),\mbox{spec}(B_0^{(n-1)} )=(\bar \lb,\bar \mu) \right]\label{condP}
 \end{aligned}\ee
 with
 $$
 \HB(\bar \lambda,\bar \mu)=
 \HB\cap \left\{ \bigl(\mbox{spec} (B^{(n)} ),\mbox{spec}(B_{}^{(n-1)} \bigr)=(\bar \lb,\bar \mu)\right\} 
 $$
 Next we compute the two probabilities in the integrand of the integral (\ref{condP}): 
%
%
\newline {(i)} The first integral equals $p(t;\bar B,B)dB$, as in (\ref{T5}). Since the Haar measure $dB$ is the product measure over all the free parameters, one will express $dB$ as the product of Haar measure $dB^{(n-1)}$ on the $(n-1)\times (n-1)$ minor and the measure $\prod_{1\leq i\leq n-1\atop{0\leq\ell\leq\beta -1}}dB_{in}^{(\ell)}dB_{nn}$ on the last row and column, remembering the expression (\ref{N}) for $ N_{n,\beta}$, thus giving:
\begin{multline}\label{cond1}
\begin{aligned}
\Prob[B_t\in dB\vert B_0=\bar B]
&=\frac{  Z^{-1}_{n,\beta}}{(1-c^2)^{N_{n,\beta}}}e^{-\frac{\beta}{2(1-c^2)}\Tr(B -c\overline{B})^2}dB\\
\\
&=\frac{  Z^{-1}_{n-1,\beta}}{(1-c^2)^{N_{n-1,\beta}}}e^{-\frac{\beta}{2(1-c^2)}\Tr(B^{(n-1)} -c\bar{B}^{(n-1)})^2}dB^{(n-1)}\\
\\
 &\quad\times\frac{  Z^{-1}_{n,\beta}  Z_{n-1,\beta}}{(1-c^2)^{(N_{n,\beta}-N_{n-1,\beta})}}e^{-\frac{\beta}{ (1-c^2)}
\left[ \sum_{1\leq i\leq n-1\atop{0\leq\ell\leq\beta -1}}(B_{in}^{(\ell)}-c\bar B_{in}^{(\ell)})^2+\frac{1}{2}(B_{nn}-c\bar B_{nn})^2\right]}    \\
 &\quad\times\prod_{i=1} ^{n-1}  \prod_{\ell=0} ^{\beta -1}dB_{in}^{(\ell)}dB_{nn}.
\end{aligned}
\end{multline}
As mentioned prior to (\ref{2.8}), one can set in (\ref{cond1}), 
\be\label{B0} dB^{(n-1)}=Z_{n-1,\beta}C^{-1}_{n-1,\beta}
|\Dt_{n-1}(\mu)|^{\beta}dU^{(n-1)}\prod_ {i=1}^{n-1} d\mu_i .\ee
In (\ref{2.22}), it was shown that upon conjugation by an appropriate matrix $U^{(n-1)}\in \mathcal{U}_{n-1}^{(\beta)}$, the matrix $B$ could be transformed in the bordered matrix $B_\bord$, as in  (\ref{2.22}) and (\ref{v}), with $(r_1u_1,\ldots,r_{n-1}u_{n-1})^\top=U^{(n-1)}v $ and $ | u_{i} | =1$.  Using the same inner-product as in the formula just preceding (\ref{1.21}), but for $n-1$-vectors, and using the associated norm $\lVert ~~\lVert$, one finds, using the above,
%
%
\be
\begin{aligned}
\lefteqn{\sum_{{1\leq i\leq n-1}\atop{0\leq \ell\leq \beta-1}}
(B_{in}^{(\ell)}-c\bar B_{in}^{(\ell)})^2
}\\&=\lVert v-c\bar v \rVert^2
\\
&=  \lVert U^{(n-1)}(v-c\bar v )\rVert^2
\\
&=   \lVert U^{(n-1)}v\rVert^2
+c^2  \lVert U^{(n-1)}\bar v\rVert^2
-2c\Re \langle U^{(n-1)}v,~U^{(n-1)}\bar v  \rangle
\\
\\
&=   \lVert U^{(n-1)}v\rVert^2
+c^2  \lVert U^{(n-1)}\bar v\rVert^2
-2c\Re \langle U^{(n-1)}v,~(U^{(n-1)}({\bar U}^{(n-1)})^{-1})({\bar U}^{(n-1)})\bar v) \rangle
\\
&= \sum_1^{n-1} r_i^2 +c^2 \sum_1^{n-1} \bar r_i^2
\\
&\hspace*{.5cm} -2c \Re \langle (  r_1  u_1,\ldots,  r_{n-1}   u_{n-1})^\top,~(U^{(n-1)}({\bar U}^{(n-1)})^{-1}) 
(\bar r_1\bar u_1,\ldots,\bar r_{n-1} \bar u_{n-1})^\top \rangle
\label{B1}\end{aligned}
\ee
%
%
Given that  $U^{(n-1)}$ is fixed and that $\det(U^{(n-1)})=1$, and since the expressions $u_i$ in (\ref{2.22}) have $ | u_{i} | =1$, the differential below can be written in terms of a product of differentials $d(r_iu_i)$ along $S^{ \beta-1 }\subset \BR^{\beta}$, expressed  in polar coordinates, thus yielding differentials involving the $r_i$'s 
and volume elements on the unit sphere $S^{\beta-1}$:  
\begin{equation}\begin{aligned}\label{B2}
\prod_{i=1}^{n-1} 
\prod_{\ell = 0}^{\beta -1}dB_{in}^{(\ell)} 
 =\prod_{i=1}^{n-1}dv_i&=\prod_{i=1}^{n-1} d((U^{(n-1)})^{-1}(  r_1  u_1,\ldots,  r_{n-1}   u_{n-1})^\top)_i 
\\&=\vert \det(U^{(n-1)})^{-\beta}\vert \prod_1^{n-1} d(r_iu_i)
\\&=\prod_1^{n-1}r_i^{\beta -1}\,dr_i\,d\Omega_i^{(\beta -1)}(u_i)~~~~~~\mbox{~~~and~~~}~~~~~~dB_{nn}=dr_n.
\end{aligned}\end{equation}
Thus all together, setting (\ref{B0}), (\ref{B1}) and (\ref{B2}) in (\ref{cond1}), we have shown that
\be\begin{aligned}
\label{cond2}
%
\lefteqn{\Prob[B^{(n)}_t\in dB\vert B^{(n)}_0=\bar B]}
\\
=&\frac{C^{-1}_{n-1,\beta}}{(1-c^2)^{N_{n-1},\beta}}e^{-\frac{\beta}{2(1-c^2)}\sum^{n-1}_{i=1}(\mu^2_i+c^2\bar\mu_i^2)}
\\ & \times|\Dt_{n-1}(\mu)|^{\beta}
e^{\frac{\beta c}{1-c^2}\Tr \left( (  U^{(n-1)}(  \bar U^{(n-1)})^{ -1}  )^{-1}\mu     U^{(n-1)}    (\bar U^{(n-1)})^{ -1}  \bar \mu \right) }
dU^{(n-1)}\prod_1^{n-1}d\mu_i
\\
 &\times\frac{  Z^{-1}_{n,\beta}  Z_{n-1,\beta}}{(1-c^2)^{(N_{n,\beta}-N_{n-1,\beta})}}e^{-\frac{\beta}{(1-c^2)}
\left(\left(\sum_1^{n-1}r_i^2+\frac{1}{2}r^2_n\right)+c^2\left(\sum_1^{n-1}\bar r_i^2+\frac{1}{2}
\bar r_n^2  \right)\right)}    \\
 &\times e^{\frac{2\beta c}{1-c^2}\Re \la (r_iu_i)_{1}^{n-1}, U^{(n-1)}(\bar U^{(n-1)})^{ -1} (\bar r_i\bar u_i)_{1}^{n -1}\ra}e^{\frac{\beta c}{1-c^2}r_n\bar r_n }\,dr_n\prod_1^{n-1}r_i^{\beta -1}\,dr_i\,d\Omega_i^{(\beta -1)}(u_i)
 \\
 =&\frac{ Z_{n,\beta}^{-1}Z_{n-1,\beta}}{(1\!-\!c^2)^{N_{n,\beta}}C_{n-1,\beta}}
 e^{-\frac{\beta}{2(1-c^2)}\sum^{n }_{ 1}(\lb^2_i+c^2\bar\lb_i^2)} e^{\frac{\beta c}{1-c^2}r_n\bar r_n }
  \vert \Dt_n(\lb)\Dt_{n-1}(\mu)\vert \, |\Dt_{n-1}(\lambda,\mu)|^{\frac{\beta}2-1}\\
 &\times~~ {\mathcal G} _{n-1}^{( \beta)} 
   \Bigl(U^{(n-1)}(\bar U^{(n-1)})^{ -1};\frac{\beta c}{1-c^2} \mu,\bar \mu; ~\frac{\beta c}{1-c^2}(r_iu_i)_{1}^{n-1},(\bar r_i \bar u_i)_{1}^{n-1}\Bigr)\\
   &\times
  dU^{(n-1)}~   \,\prod_1^{n-1}d\Omega_i^{(\beta -1)}(u_i)
   \prod_1^{n-1}  d\mu_i \prod_1^{n }  d\lb_i
.
\end{aligned}\ee
In the last equality we have used identities (\ref{2.16}) and (\ref{2.15}) and the definition (\ref{Gint}) of ${\mathcal G}_{n-1}^{(\beta)}$.

\noindent {(ii)} The second probability in (\ref{condP}) takes on the following  value:
\be
\label{cond2'}\Prob   \left[B^{(n)}_0\in d\bar B \big\vert (\mbox{spec} (B^{(n)}_0),\mbox{spec}(B^{(n-1)}_0)=(\bar \lb,\bar \mu) \right]=d\bar U^{(n-1)}\frac{\prod_1^{n-1}d\Omega_i^{(\beta -1)}(\bar u_i)}{(\vol (S^{\beta-1}))^{n-1}}.
\ee
Indeed, the probability (\ref{cond2}), when $t\to \infty$ ( which amounts to letting $c\rg 0$), tends to the invariant measure for the Dyson Brownian motion; instead of the usual representation in the variables $(\lb_i,U^{(n)})$, this gives the expression of the GUE-probability in the variables $(\lb_i,\mu_j,u_k, U^{(n-1)})$   :
\be\begin{aligned}
\label{f}
\Prob[B^{(n)} \in dB]&=Z^{-1}_{n,\beta}e^{-\frac \beta 2\Tr B^2}dB
\\&=\lim_{t\to \infty}\Prob[B^{(n)}_t\in dB\vert B^{(n)}_0=\bar B]
\\
&=f_{n,\beta}( \lambda,\mu)
dU^{(n-1)}~   \,\prod_1^{n-1}d\Omega_i^{(\beta -1)}(u_i)
   \prod_1^{n-1}  d\mu_i \prod_1^{n }  d\lb_i
   \end{aligned}\ee
with (using ${\mathcal G}_{n-1}^{(\beta)}(U;0,\bar \mu;0,(\bar r_i\bar u_i)_1^{n-1})=1$)
\be\label{f'}
f_{n,\beta}( \lambda,\mu):=
\frac{ Z_{n,\beta}^{-1}Z_{n-1,\beta}}{ C_{n-1,\beta}}
 e^{-\frac{\beta}{2 }\sum^{n }_{ 1} \lb^2_i } 
 \vert \Dt_n(\lb)\Dt_{n-1}(\mu)\vert \, |\Dt_{n-1}(\lambda,\mu)|^{\frac{\beta}2-1}
 .\ee
This also shows that $ \HB(  \lambda,  \mu)\simeq  \UR_{n-1}^{(\beta)}\times \left(S^{(\beta-1)}\right)^{n-1}$. Using (\ref{f}), one checks that the conditional probability equals
 \be\begin{aligned}
\label{}
 \lefteqn{\Prob   \left[B^{(n)}_0\in d\bar B \big\vert (\mbox{spec} (B^{(n)}_0),\mbox{spec}(B^{(n-1)}_0)=( \bar \lb, \bar \mu) \right]}\\
 &=
\frac{\Prob   \left[B^{(n)}_0\in d\bar B ,~ (\mbox{spec} (B^{(n)}_0),\mbox{spec}(B^{(n-1)}_0)\in (d\bar \lb,d\bar \mu) \right]}
{\Prob   \left[ (\mbox{spec} (B^{(n)}_0),\mbox{spec}(B^{(n-1)}_0)\in(d\bar \lb,d\bar \mu) \right]}
\\
&=\frac{f_{n,\beta}( \bar\lambda,\bar\mu)
d\bar U^{(n-1)}~   \,\prod_1^{n-1}d\Omega_i^{(\beta -1)}(\bar u_i)
   \prod_1^{n-1}  d\bar\mu_i \prod_1^{n }  d\bar\lb_i
   }{f_{n,\beta}( \bar\lambda,\bar\mu)\prod_1^{n-1}  d\bar\mu_i \prod_1^{n }  d\bar\lb_i {\dis\int_{ \UR_{n-1}^{(\beta)}\times \left(S^{(\beta-1)}\right)^{n-1}  }}
d\bar U^{(n-1)}~   \,\prod_1^{n-1}d\Omega_i^{(\beta -1)}(\bar u_i)
}
   \\
&=\frac{ 
d\bar U^{(n-1)}~   \,\prod_1^{n-1}d\Omega_i^{(\beta -1)}(\bar u_i)
   }{    \left(\vol(S^{(\beta-1)}\right)^{n-1}  
},
\end{aligned}\ee
confirming expression (\ref{cond2'}).
%
Setting (\ref{cond2}) and (\ref{cond2'}) in (\ref{condP}) and using the integral (\ref{1.21}) and the identification just after (\ref{f'}), one computes:
\be\begin{aligned}
\lefteqn{p_{\lb,\mu}(t;(\bar \lb,\bar \mu),(  \lb,  \mu))d\lb d\mu
}\\
=&\int_{B _t\in \HB(  \lambda,  \mu)}\int_{\bar B  \in \HB(\bar \lambda,\bar \mu) }
\Prob \left[B^{(n)}_t\in dB~\big\vert ~B^{(n)}_0=\bar B\right]
\\
&~~ \Prob   \left[B^{(n)}_0\in d\bar B \big\vert (\mbox{spec} (B_0^{(n)} ),\mbox{spec}(B_0^{(n-1)} )=(\bar \lb,\bar \mu) \right]
\\
=&\int_{\left(\UR_{n-1}^{(\beta)}\times \left(S^{(\beta-1)}\right)^{n-1}\right)^2}
 \frac{ Z_{n,\beta}^{-1}Z_{n-1,\beta}C_{n-1,\beta}^{-1}}{(1\!-\!c^2)^{N_{n,\beta}}\vol(S^{(\beta-1)})^{n-1}}
 e^{-\frac{\beta}{2(1-c^2)}\sum^{n-1}_{i=1}(\lb^2_i+c^2\bar\lb_i^2)} e^{\frac{\beta c}{1-c^2}r_n\bar r_n }
\\
 &\times~~ {\mathcal G} _{n-1}^{( \beta)} 
   \Bigl(U^{(n-1)}(\bar U^{(n-1)})^{ -1};\frac{\beta c}{1-c^2} \mu,\bar \mu; ~\frac{\beta c}{1-c^2}(r_iu_i)_{1}^{n-1},(\bar r_i \bar u_i)_{1}^{n-1}\Bigr)
   dU^{(n-1)}d\bar U^{(n-1)}\\
   &\times
  ~   \vert \Dt_n(\lb)\Dt_{n-1}(\mu)\vert \, |\Dt_{n-1}(\lambda,\mu)|^{\frac{\beta}2-1} \,\prod_1^{n-1}d\Omega_i^{(\beta -1)}(u_i)
  \prod_1^{n-1}d\Omega_i^{(\beta -1)}(\bar u_i)
   \prod_1^{n-1}  d\mu_i \prod_1^{n }  d\lb_i
\\
=& 
  \frac{ \hat Z_{n,\beta}^{-1}}{(1\!-\!c^2)^{N_{n,\beta}} }
 e^{-\frac{\beta}{2(1-c^2)}\sum^{n-1}_{i=1}(\lb^2_i+c^2\bar\lb_i^2)} e^{\frac{\beta c}{1-c^2}r_n\bar r_n }
  \vert \Dt_n(\lb)\Dt_{n-1}(\mu)\vert \, |\Dt_{n-1}(\lambda,\mu)|^{\frac{\beta}2-1}  d\mu  d\lb \\
 &\!\!\times\!\! \int_{  \left(S^{(\beta-1)}\right)^{2n-2}  }
 G _{n-1}^{( \beta)} 
   \Bigl( \frac{\beta c}{1\!-\!c^2} \mu,\bar \mu; ~\frac{\beta c}{1\!-\!c^2}(r_iu_i)_{1}^{n-1},(\bar r_i \bar u_i)_{1}^{n-1}\Bigr)
   \,\prod_1^{n-1}d\Omega_i^{(\beta -1)}(u_i)
  d\Omega_i^{(\beta -1)}(\bar u_i),
\end{aligned}\ee
where we used the translation invariance of $dU^{(n-1)}$ and $\vol({\mathcal U}^{(\beta)}_{n-1})=1$; also one checks the value of the constant $\hat Z_{n,\beta}^{-1}$ to be the one  given in footnote \ref{footZhat}. This establishes formula (\ref{1.24}) for the transition probability of the $(\lambda_t, \mu_t)$-process.
 

The statements concerning the invariant measures,  (\ref{1.28}) and 
(\ref{1.29}) follow immediately from~(\ref{1.20}), (\ref{1.21}),  
(\ref{1.24}), by letting $t\rg \infty$ in  the transition 
probability. 
This concludes  the proof of the formulae for the transition probabilities 
(\ref{OU-tr}), (\ref{1.20}) and invariant measure~(\ref{1.29}), appearing in 
Theorems~\ref{Th:Dyson} and~\ref{Th:Dyson2}.
\end{proof}

\begin{remark}
 The diffusion equation (\ref{1.12'}), which will be established 
in section \ref{SDE1}, can also be used to confirm the form of the invariant 
measure, at least for $\beta=2$. On general grounds, the density of the 
invariant measure, namely
\begin{equation} 
 I_{\lb\mu}(\lb,\mu):=C e^{-\frac{\beta}{2}\sum_1^n\lb^2_i}|\Dt_n(\lb)\Dt_{n-1}(\mu)|~|\Dt_n(\lb,\mu)|^{ \frac{\beta}{2}-1 }
,\label{3.10'}\end{equation} 
is a null vector of the forward equation, i.e.
\begin{equation} 
\AR^{\top} I_{\lb\mu}(\lb,\mu)= \bigl(\AR_{\lb}^\top  +\AR_{\mu}^\top +\AR^\top_{\lb \mu} \bigr) I_{\lb\mu}(\lb,\mu)=0,
\label{3.12}\end{equation} 
with $\mathcal{A}$ defined in (\ref{1.12'}).
For $\beta =2$, more is true; namely 
\begin{align} 
\AR_{_{\lb}}^\top (\lb)I_{\lb\mu}&=\frac{n(n-1)}{2}I_{\lb\mu},&
\AR_{_{\mu}}^\top (\mu)I_{\lb\mu}&=\frac{n(n-1)}{2}I_{\lb\mu}
.\label{3.13}\end{align} 
Once this is shown, 
 it follows that
\begin{equation}
\AR^\top_{\lb\mu}  I_{\lb\mu}(\lb,\mu)=-n(n-1) I_{\lb\mu}(\lb,\mu).
\end{equation}
So it suffices to prove (\ref{3.13}). 
  First observe that $\Dt_n(\lb)$ and $\Dt_{n-1}(\mu)$ are harmonic functions, i.e.
\begin{align*}
\sum_1^n\left(\frac{\partial}{\partial\lb_i}\right)^2\Dt_n(\lb)&=0,
&\sum_1^{n-1}\left(\frac{\partial}{\partial\mu_i}\right)^2\Dt_{n-1}(\mu)&=0,
\end{align*}
and also homogeneous functions so that acted upon by the Euler operators,
\begin{align*}
\sum_1^n\lb_i \frac{\partial}{\partial\lb_i}\Dt_n(\lb)&=\frac{n(n-1)}{2}\Dt_n(\lb),\\
\sum_1^{n-1}\mu_i\frac{\partial}{\partial\mu_i}\Dt_{n-1}(\mu)&=\frac{(n-1)(n-2)}{2}\Dt_{n-1}(\mu).
\end{align*}
Now compute from (\ref{1.5'})  and (\ref{3.10'})  that (remember $\Phi_n(\lb):=e^{-\frac{1}{2}\sum_1^n\lb^2_i}\\ \vert\Dt_n(\lb)\vert$)
\begin{align*}
\AR_{_{\lb}}^\top (\lb)I_{\lb\mu}
 &=\frac{1}{2}\sum_1^n \frac{\partial}{\partial\lb_i}(\Phi_n(\lb))^2 
\frac{\partial}{\partial\lb_i}\frac{\Dt_n(\lb)\Dt_{n-1}(\mu)  e^{-\sum_1^n\lb^2_j}}{(\Phi_n(\lb))^2}\\
&=-\frac{\Dt_{n-1}(\mu)}{2}\sum_1^n\frac{\partial}{\partial\lb_i}e^{-\sum_1^n\lb^2_j}\frac{\partial}{\partial\lb_i}\Dt_n(\lb)\\
\\
&= -\frac{1}{2}e^{-\sum_1^n\lb^2_j}\Dt_{n-1}(\mu)\sum_{i=1}^n\left(\frac{\partial^2}{\partial\lb_i^2} -2 \lb_i\frac{\partial}{\partial\lb_i}\right)\Dt_n(\lb) \\
\\
&=
 \frac{n(n-1)}{2}I_{\lb\mu}
\intertext{and also that}
\AR_{\mu}^\top (\mu)I_{\lb\mu}
&=
 \frac{1}{2} \sum^{n-1}_1 \frac{\partial}{\partial \mu_i} (\Phi_{n-1} (\mu))^2 
  \frac{\partial}{\partial \mu_i} \frac{\Delta_n (\lb) \Delta_{n-1} (\mu)   
  e^{-\sum^n_1\lb^2_j}}{(\Phi _{n-1} (\mu))^2}
 \\
 &= - \frac{1}{2} \Delta_n (\lb) e^{ -\sum^n_1 \lb^2_i} \sum^{n-1}_1 
\frac{\partial}{\partial \mu_i} e^{ \sum^{n-1}_1 \mu^2_i}  
\frac{\partial}{\partial \mu_i} \left(\Delta_{n-1}(\mu)
e^{  -\sum^{n-1}_1 \mu^2_i}\right)
 \\
  &=
  - \frac{1}{2} \Delta_n (\lb) e^{ -\sum^n_1 \lb^2_i}  \left(\sum^{n-1}_1 (\frac{\partial^2}{\partial \mu^2_i}- 2 
   \mu_i \frac{\partial}{\partial \mu_i} )  \Delta_{n-1} (\mu)  - 2(n\!-\!1) \Delta_{n-1} (\mu)  \right)
  \\
 &=
   \frac{1}{2} \Delta_n (\lb) e^{  -\sum^n_1 \lb^2_i}  ( (n-1)(n-2)+2(n-1))\Delta_{n-1}(\mu) \\
   &= \frac{n(n-1)}{2} I_{\lb\mu}.
\end{align*}
This ends the proof of identities (\ref{3.13}).
\end{remark}

%
%

\section{It\=o's Lemma and Dyson's Theorem}\label{SDE1}
\label{sec:ito}
To fix notation we repeat some well known facts from stochastic 
calculus in a way that will be useful later. 
Given a diffusion $X_t\in \BR^n$, given by the 
SDE\footnote{The subscript $t$ in $X_t$ and $B_t$ will often be omitted.}
\begin{equation}
dX_t=\sigma (X_t)db_t+a(X_t)dt,
\label{3.1}
\end{equation}
where $db_t$ is a vector of independent standard Brownian motions, 
where $x$, $a(x)\in \BR^n$ and $\sigma(x)$ an $n\times n$ matrix. 
Then the generator of this diffusion is given by
\begin{equation}
\mathcal{A}=\frac12 \sum_{i,j}(\sg\sg^\top)_{ij}(x) \frac{\partial^2}{\partial x_i\partial x_j}+\sum_ia_i(x) 
 \frac{\partial }{\partial x_i }
,\end{equation}
and, by straight forward verification, 
\begin{equation}\label{3.2}
\begin{split}
(\sg\sg^\top)_{ij}(x)& =\mathcal{A}(x_ix_j)  
-x_i \mathcal{A}(x_j)-x_j \mathcal{A}(x_i)
=
\left(\frac{dX_idX_j}{dt}\right) (x)
\\
a_i(x) &=\mathcal{A}x_i.
\end{split}
\end{equation}
The transition density $p(t,\bar x,x)$ is a solution of the forward equation (in $x$)
\begin{equation} 
\frac{\partial p}{\partial t}=\mathcal{A}^\top p.
\label{3.3}\end{equation}
Moreover for a function 
 $g:\BR^n \mapsto \BR^p$ with $g\in \mathcal{C}^2$,
the SDE for $Y_t=g(X_t)$ has the form
 \begin{equation} 
 dY_k=\sum_i \frac{\partial g_k}{\partial x_i}dX_i+\frac12
 \sum_{i,j}\frac{\partial^2 g}{\partial x_ix_j}dX_idX_j=
\sum_j(\sum_i \frac{\partial g_k}{\partial x_i}\sg_{ij}) db_j
+h_k dt ,\label{3.5}\end{equation}
 for $k=1$, \dots, $p$ and for some function $h_k$; 
i.e., the local martingale part only depends on first derivatives of $g$. 
This follows from the standard multiplication rules of stochastic
calculus ($dtdt=0$, $dtdb=0$ and $db_idb_j = \delta_{ij} dt$):
\begin{align}
dX_i\, dX_j &= (a_i\, dt + \sum_{\ell = 1}^n \sigma_{i\ell}\, db_\ell)
(a_j\, dt + \sum_{k = 1}^n \sigma_{jk}\, db_k)
\\
&= (\sum_{\ell = 1}^n \sigma_{i\ell}\, db_\ell)
(\sum_{k = 1}^n \sigma_{jk}\, db_k) = (\sigma \sigma^\top )_{ij}\, dt.
\end{align}


More details can be found in any book on stochastic calculus,
for example~\cite{McK} or~\cite{Oksendal}.
As a warm-up exercise, we first prove Dyson's original result, 
namely the formulae for the SDE and for the generator of 
Theorem~\ref{Th:Dyson}, including some consequences.

\begin{proof}[Proof of~\eqref{1.4} and~\eqref{1.5'} in Theorem \ref{Th:Dyson}.]
   The Dyson process is invariant under conjugation by $U\in \UR_n^{(\beta)}$; 
to be precise from (\ref{DOUgen}),
$$
p (t; U\bar B U^{-1}, U B U^{-1})=p (t;  \bar B  ,   B  )
.$$
Therefore, we are free, at any fixed choice of $t$, to reset 
$$
 B(t)\mapsto UB(t)U^{-1}, ~~\text{for any}~U\in   \UR_n^{(\beta)}
 .$$
At any given time $t$, diagonalize the matrix $B$ to yield $\diag (\lb_1,\ldots,\lb_n)$ and consider the perturbation
$$
\diag (\lb_1,\ldots,\lb_n)+[{ d B}_{ij}],
$$
where 
one defines the $n\times n$ matrix, for $1\leq i<j\leq n$, 
\be[{ d B}_{ij}]:= \begin{pmatrix}
  \ddots  \\
    & 0& \cdots &  dB_{ij}^{(0)}+\sum_1^{\beta-1}  dB_{ij}^{(\ell)}e_\ell \\
 &\vdots & \ddots & \vdots \\
   &  dB_{ij}^{(0)}-\sum_1^{\beta-1}  dB_{ij}^{(\ell)}e_\ell& \cdots & 0\\
 &   && &\ddots&
  \end{pmatrix}
\label{perturb}\ee
and, for $i = 1$, \dots, $n$,
$$[{ d B}_{ii}]:=\diag(0,\ldots,d B_{ii},\ldots,0),
$$
with, by (\ref{1.3}), 
 \begin{align} 
 dB_{ij}^{(\ell)}\, dB_{i'j'}^{(\ell')}&=\dt_{ii'}\dt_{jj'}\dt_{\ell\ell'}\frac{dt}{\beta}
&
 dB_{ii}\,dB_{jj}&= 2\dt_{ij}\frac{dt}{\beta}&
dB_{ij}^{(\ell)}\, dB_{kk}&=0
.  \label{3.8}\end{align}
 Remember by Ito's formula~\eqref{3.5}, one only needs to keep track of at most second order changes of the arguments. 
 
 Thus for \emph{non-diagonal perturbations} ($i\neq j$), one checks\footnote{Remember, for $\beta=4$, the determinant is defined in (\ref{quatdet}).}
\begin{equation} 
 \label{3.10}
\begin{split}
 0&=\det \Bigl(\diag (\lb_1,\ldots,\lb_n)+[{ d B}_{ij}]-\lb I\Bigr)\Bigr|_{{\lb \mapsto \lb_\al+d\lb_\al }}\\
& = 
 \left.\prod_{{\ell=1}\atop{\ell \neq i,j}}^{n} ( \lb_\ell-\lb)
 \left( (\lb-\lb_i)(\lb-\lb_j)-\sum _{\ell=0}^{\beta-1}(dB^{(\ell)}_{ij})^2 \right)\right|_{{\lb \mapsto \lb_\al+d\lb_\al }}\\
 &=
\begin{cases}
  (\text{non-zero function})\times d\lb_\al,  & \text{for $\alpha\neq i$, $j$,}
  \\
  (\text{non-zero function})\times  ((\lb_i-\lb_j)d\lb_i -\sum _{\ell=0}^{\beta-1}(dB^{(\ell)}_{ij})^2)
 & \text{for $\alpha= i$,}
  \\
   (\text{non-zero function}) \times ((\lb_j-\lb_i)d\lb_j -\sum _{\ell=0}^{\beta-1}(dB^{(\ell)}_{ij})^2)
& \text{for $\alpha= j$,}
\end{cases}
\end{split}
\end{equation}
showing that an off-diagonal perturbation of the diagonal matrix $B(t)=\diag (\lb_1(t),\ldots,\lb_n(t))$ yields
$$
 d\lb_i=  \frac{\sum _{\ell=0}^{\beta-1}(dB^{(\ell)}_{ij})^2}{\lb_i-\lb_j},~~
d\lb_j=  \frac{\sum _{\ell=0}^{\beta-1}(dB^{(\ell)}_{ij})^2}{\lb_j-\lb_i},~~\text{and}~d\lb_\al=0 \text{  for  } \al \neq i,j
.$$
For \emph{diagonal perturbations} ($i=j$), one finds
\begin{multline*}
 \det \Bigl( \diag (\lb_1,\ldots,\lb_n)+[{ d B}_{ii}]-\lb I \Bigr)\Bigr|_{{\lb \mapsto \lb_\al+d\lb_\al }}  
=  \prod_{{\ell=1}\atop{\ell \neq i}}^{n} ( \lb_\ell-\lb)
 \left( \lb_i+dB_{ii}-\lb\right)\Bigr|_{{\lb \mapsto \lb_\al+d\lb_\al }}
  \\
=
\begin{cases}
 (\text{non-zero function}) \times d\lb_\al ,&   \text{for $\al\neq i$} \\
  (\text{non-zero function}) \times (dB_{\al\al} -d\lb_\al ) ,&   
\text{for $\al= i  $.}
\end{cases}
 \end{multline*}
and thus 
$$  d\lb_\al =0  ~   \text{for}~ \al\neq i ~   \text{and}~
 d\lb_\al  =dB_{\al\al}~   \text{for}~ \al = i
$$
Then summing up all the perturbations, one finds the SDE~\eqref{1.4}
 announced in Theorem~\ref{Th:Dyson}, using 
$dB_{ii}=-B_{ii}dt+\sqrt{\frac{2}{\beta}}db_{ii}=-\lb_i dt+\sqrt{\frac{2}{\beta}}db_{ii}$ and formula (\ref{3.8}),
\begin{align*}
d\lb_i&=\left(dB_{ii}+\sum_{j\neq i}
 \frac{\sum _{\ell=0}^{\beta-1}(dB^{(\ell)}_{ij})^2}{\lb_i-\lb_j}  \right)
\\
 &=\left(-\lb_i  +\sum_{j\neq i}\frac{1}{\lb_i-\lb_j}\right)dt+\sqrt{\frac{2}{\beta}}db_{ii}, &&\text{for $i=1$, \dots, $n$.}
\end{align*}
Then translating the SDE into the generator of the diffusion on 
$(\lb_1,\ldots,\lb_n)$, one finds, by (\ref{3.2}), that
\begin{equation}
\mathcal{ A}_{_{\lb}}= \sum_1^n\left(\frac1\beta  \frac{\partial^2}{\partial \lb_i^2}+
 \Bigl(-\lb_i  +\sum_{j\neq i}\frac{1}{\lb_i-\lb_j}\Bigr)
  \frac{\partial }{\partial \lb_i }\right),
    \label{3.13'}\end{equation}
  and thus 
\begin{align*}
 \mathcal{ A}_\lb^\top&=
\sum_1^n\left(\frac1\beta   \frac{\partial^2}{\partial \lb_i^2}-
  \frac{\partial }{\partial \lb_i }\bigl(-\lb_i  +\sum_{j\neq i}\frac{1}{\lb_i-\lb_j}\bigr)
 \right)\\
&=\frac{1}{\beta}
\sum^n_{i=1}\frac{\partial}{\partial\lb_i}\left(\Phi_n (\lb)\right)^{\beta}
\frac{\partial}{\partial \lb_i}\frac{1}{(\Phi_n (\lb))^{\beta}} 
, \end{align*}
with $\Phi_n (\lb)$ as in~(\ref{1.28}), confirming formula~(\ref{1.5'}) in 
Theorem~\ref{Th:Dyson}. Finally~$\mathcal{ A}_{\DOU} \lb_i= \mathcal{ A}_{_{\lb}}\lb_i$, mentioned in (\ref{restr}), follows from the fact that 
the generator $\mathcal{ A}_{\DOU}$ restricted to the 
functions~$(\lb_1,\ldots\lb_n)$ equals $ \mathcal{ A}_{_{\lb}}$, as a consequence of (\ref{3.1}) to (\ref{3.3}); 
of course, this holds for the spectrum of every principal minor of the 
matrix $B$. 
\end{proof}

\section{SDE for the Dyson process on the spectra of two consecutive minors}\label{SDE2}

In this section we prove the formulas (\ref{SDE-lambda0}) for the $\lb$- and $\mu$-SDE's, together with the generator (\ref{1.12'}).

 \begin{proof}[Proof of SDE (\ref{SDE-lambda0}) and generator~\eqref{1.12'} 
in Theorem~\ref{Th:Dyson2}]  Using the same idea as in the proof 
of~\eqref{1.4} and~\eqref{1.5'} in the  Section~\ref{sec:ito},
 we choose, at time $t$, to conjugate the 
matrix~$B$ so as to have the form $B_\bord$ of (\ref{2.22}) and let the 
matrix $B_\bord$ evolve according to the Dyson process. 
 We will consider only the first order effects on the $\lb$'s and ignore second order effects. 
 
 At first, we need to compute the (first order) variation of the $\lb_\al$'s as a function of the (first order) variation of the entries:

\noindent \emph{Case 1:\/} Consider the perturbation of $B_\bord$, 
using the notation (\ref{perturb}) for $[dB_{ij}]$, namely
\begin{equation} 
B_\bord + [dB_{ij}], \text{ for $1\leq i<j\leq n-1$}
.\label{3.18}\end{equation}
Up to first order, one must compute the effect of the perturbation on each 
of the eigenvalues $\lb_\al$, by explicitly computing the characteristic 
polynomial of the bordered matrix (\ref{2.22}) with the extra non-diagonal perturbation; then, by neglecting the second order terms in $dB$, one finds\footnote{Notice that $ 2\Re d  B_{ij}^\ast u_i u_j^\ast=
d  B_{ij}^\ast u_i  u_j^\ast+   {(d B_{ij}^\ast u_i  u_j^\ast} )^\ast
=2\sum_{\ell=0}^{\beta-1} dB_{ij}^{(\ell)} (u_i  u_j^\ast)^{(\ell)} $, using $(ab)^\ast=  b^\ast   a^\ast$. Remember, for $\beta=4$, quaternion multiplication  
does not commute and for the determinant of a matrix, use formula (\ref{2.24'}).}:
\begin{multline}
  0=\det(B_\bord+[dB_{ij}]-\lb I)\Bigr|_{\lb \mapsto \lb_\al+d\lb_\al } 
= \prod_{\ell=1}^{n-1} ( \mu_\ell-\lb)\\
\left(\sum_1^{n-1} \frac{r_\ell^2}{\lb-\mu_\ell}+r_n-\lb+\frac{2r_ir_j}{(\lb-\mu_i)(\lb-\mu_j)}\sum_{\ell=0}^{\beta-1} dB_{ij}^{(\ell)} (u_i  u_j^\ast)^{(\ell)}\right) \Bigr|_{\lb \mapsto \lb_\al+d\lb_\al }.
\end{multline}
Setting $\lb\mapsto\lb_\al+d\lb_\al $ in this expression, shows that the product $\prod_{\ell=1}^{n-1} ( \mu_\ell-\lb_\al-d\lb_\al)$ is of the form (non-zero-function)$+\text{(function)}\times d\lb_\al$, whereas the second part gives, by Taylor expanding in $\lb_\al$,  keeping in the expression first order terms only, evaluated by (\ref{2.28}), and noticing that the $0^{th}$-order term vanishes (again using (\ref{2.28})), we find
$$
-\frac{P_n'(\lb_\al)}{P_{n-1}(\lb_\al)}d\lb_\al+\frac{2r_ir_j}{(\lb_\al-\mu_i)(\lb_\al-\mu_j)}
\sum_{\ell=0}^{\beta-1} dB_{ij}^{(\ell)} (u_i  u_j^\ast)^{(\ell)}=0
.$$
Finally, adding up the first order contributions from all the perturbations $dB_{ij}^{(0)}+\sum_1^{\beta-1}  dB_{ij}^{(\ell)}$ , with $1\leq i<j\leq n-1$, yields
\be
 d\lb_\al =
 \frac{P_{n-1}(\lb_\al)}{P_n'(\lb_\al)}
 \sum_{1\leq i< j\leq n-1}
 \frac{2r_ir_j}{(\lb_\al-\mu_i)(\lb_\al-\mu_j)}
\sum_{\ell=0}^{\beta-1} dB_{ij}^{(\ell)} (u_i u_j^\ast)^{(\ell)}
.\label{3.19}\ee

\noindent \emph{Case 2:\/} For the perturbation $[dB_{ii}]$, with $i = 1$, \dots, 
$n-1$, again neglecting the second order terms,
\begin{align*}
  0&=\det(B_\bord+[dB_{ii}]-\lb I)\Bigr|_{\lb \mapsto \lb_\al+d\lb_\al } 
\\
&= \prod_{\ell=1}^{n-1} ( \mu_\ell+\dt_{\ell i}dB_{ii}-\lb_\al-d\lb_\al)\\
&\quad
\left(\sum_{\ell=1}^{n-1} \frac{r_\ell^2}{\lb_\al+d\lb_\al-\mu_\ell-\dt_{\ell i}dB_{ii}}+(r_n+\dt_{in}dB_{ii})-\lb_\al-d\lb_\al
\right)
\end{align*}
Upon expanding this expression as a function of $\lb_\al$, $\mu_i$ up to first order, noticing as before that the first part does not matter, and using again (\ref{2.28}), this leads to
 \begin{align}
  -\frac{P_n'(\lb_\al)}{P_{n-1}(\lb_\al)}d\lb_\al+
 \frac{r_i^2}{(\lb_\al-\mu_i)^2}
 dB_{ii} &=0 &&\text{  for  $i = 1$, \dots, $n-1$,}\\
-\frac{P_n'(\lb_\al)}{P_{n-1}(\lb_\al)}d\lb_\al+ dB_{nn}&=0
&& \text{  for  $i=n$,}
\end{align}
and thus summing up all the contributions coming from the $dB_{ii}$ for $i = 1$, \dots, $n$, one finds
\be
d\lb_\al=\frac{P_{n-1}(\lb_\al)}{P_n'(\lb_\al)}
\left(\sum_{i=1}^{n-1}\frac{r_i^2}{(\lb_\al-\mu_i)^2}
 dB_{ii}  +dB_{nn}\right)
\label{3.20}\ee

\noindent \emph{Case 3:\/} For the perturbation $[dB_{in}]$, 
$i=1$, \dots, $n-1$,
\begin{align*}
  0&=\det(B_{\text{\tiny bord}}+[dB_{in}]-\lb I)\Bigr|_{\lb \mapsto \lb_\al+d\lb_\al } 
\\
&= \prod_{\ell=1}^{n-1} ( \mu_\ell-\lb)\\
&\quad
\left(\sum_{k=1}^{n-1} \frac{r_k^2+ r_i(u_i\,d   B_{in}^\ast+  u_i^\ast\, dB_{in} )
\dt_{ik}}{\lb-\mu_k}+r_n-\lb\right) \Bigr|_{\lb \mapsto \lb_\al+d\lb_\al }
.\end{align*}
Then, using $u_i  { dB_{in}^\ast}+dB_{in}  u_i ^\ast=2\sum_{\ell=0}^{\beta-1} u_i^{(\ell)}dB_{in}^{\ell)}$, using formula (\ref{2.28}) and finally summing up over all  perturbations of the last row and column ($1\leq i\leq n-1$) yields
\begin{equation}
d\lb_\al=\frac{P_{n-1}(\lb_\al)}{P_n'(\lb_\al)}
 \sum_{i=1}^{n-1}\frac{2r_i \sum_{\ell=0}^{\beta-1} u_i^{(\ell)}dB_{in}^{(\ell)}}{ \lb_\al-\mu_i  }
.\label{3.21}
\end{equation}
Then {\em summing up the three contributions} (\ref{3.19}), (\ref{3.20}) and (\ref{3.21}) gives us the total first order contribution to $d\lb_\al$:
$$
d\lb_\al=\frac{P_{n-1}(\lb_\al)}{P_n'(\lb_\al)}
\left\{
\begin{aligned}
 &\sum_{1\leq i< j\leq n-1}
 \frac{2r_ir_j}{(\lb_\al-\mu_i)(\lb_\al-\mu_j)}
\sum_{\ell=0}^{\beta-1} dB_{ij}^{(\ell)} (u_i u_j^\ast)^{(\ell)}
\\ 
&+  \sum_{i=1}^{n-1}\frac{r_i^2}{(\lb_\al-\mu_i)^2}
 dB_{ii}  +dB_{nn}
 +\sum_{i=1}^{n-1}\frac{2r_i \sum_{\ell=0}^{\beta-1} u_i^{(\ell)}dB_{in}^{(\ell)}}{ \lb_\al-\mu_i  }
 \end{aligned}\right\}
.$$
We now set the SDE's (\ref{1.3}) for the $dB_{ii},~dB_{ij}^{(\ell)}$ into the equation obtained above, thus yielding, by (\ref{3.5}),%
\begin{multline}
\label{3.22}
d\lb_\al=F_\al^{(n)} (\lb) \,dt\, + \\
\sqrt{\frac{2}{\beta}}\frac{P_{n-1}(\lb_\al)}{P_n'(\lb_\al)}
\left\{
\begin{aligned}
 &\sum_{1\leq i< j\leq n-1}
 \frac{ \sqrt{2} ~r_ir_j}{(\lb_\al-\mu_i)(\lb_\al-\mu_j)}
\sum_{\ell=0}^{\beta-1}  (u_i  u_j^\ast)^{(\ell)}db_{ij}^{(\ell)}
\\ 
&+  \sum_{i=1}^{n-1}\left(\frac{r_i }{ \lb_\al-\mu_i }\right)^2
 db_{ii}  +db_{nn}
  \\
& +\sqrt{2}\sum_{i=1}^{n-1}\frac{ r_i}{ \lb_\al-\mu_i  } \sum_{\ell=0}^{\beta-1} u_i^{(\ell)}db_{in}^{(\ell)}
 \end{aligned}\right\},
\end{multline}
for some function $F_\al^{(n)} (\lb)$ to be determined later.
Notice that in $\BR,~\BC$ and $\BH$, the norm $|v|$ satisfies $|vw|=|v|. |w|$ and $|v|=|  v^\ast |$.
Therefore, when $|u_i|=1$, we also have $|u_i  u_j^\ast|=1$, implying that
\begin{equation*}
d\tilde b_{in}:=\sum_{\ell=0}^{\beta-1} u_i^{(\ell)}db_{in}^{(\ell)}
\quad\text{ and }\quad
d\tilde b_{ij}:=\sum_{\ell=0}^{\beta-1}  (u_i  u_j^\ast)^{(\ell)}db_{ij}^{(\ell)}
\end{equation*}
are both standard Brownian motions on the sphere $S^{\beta-1}$; since they are different linear combinations, they are independent standard Brownian motions, and independent of $db_{ii},~1\leq i\leq n$. 
This is precisely formula~(\ref{SDE-lambda0}) of Theorem~\ref{Th:Dyson2}, namely
\begin{multline}\label{sde}
 d\lb_{\al}= F_\al^{(n)} (\lb) dt  +  
 \sqrt{\frac{2}{\beta}}
\frac{P_{n-1}(\lb_{\al})}{P'_n(\lb_{\al})}
 \\  \times \left(\sum_{1\leq i<j\leq n-1}\!\frac{\sqrt{2}~r_ir_j~\tilde{d  b_{ij}}}{(\lb_{\al}\!-\!\mu_i)(\lb_{\al}\!-\!\mu_j)}+\sum^{n-1}_{i=1}\frac{r^2_idb_{ii}}{(\lb_{\al}\!-\!\mu_i)^2}\right.
 \left.+\sum_{i=1}^{n-1}\frac{\sqrt{2}~r_i \tilde{d  b_{in}}}{\lb_{\al}\!-\!\mu_i}+db_{nn}\!\right)
. 
\end{multline} 
The SDE for the Dyson process induced on the $(n-1)\times (n-1) $ upper-left 
minor is given by the first formula of Theorem~\ref{Th:Dyson} with 
$n\mapsto n-1$ and $\lb\mapsto \mu$, yielding the formula 
in~(\ref{SDE-lambda0}). 
Therefore the product of the SDEs in (\ref{SDE-lambda0}), together with identity (\ref{r}) yields
\begin{equation} 
 \frac {d\lb_id\mu_j}{dt} =\frac{2}{\beta}\frac{P_{n-1}(\lb_i)}{P_n'(\lb_i)}\left(\frac{r_j }{\lb_i-\mu_j}\right)^2 
 =-\frac{2}{\beta}\frac{1}{(\lb_i-\mu_j)^2}\frac{P_{n-1}(\lb_i)P_n(\mu_j)}{P_n'(\lb_i)P_{n-1}'(\mu_j)}  .
\label{3.23}\end{equation}
Moreover,
\begin{align*}
\frac{d\lb_\al d\lb_\ga}{dt}&=
\mathcal{ A}_\DOU (\lb_\al \lb_\ga) 
- \lb_\al \mathcal{ A}_\DOU(\lb_\ga)
-\lb_\ga\mathcal{ A}_\DOU(\lb_{\al})\\
&=
\mathcal{ A}_{_{\lb}}(\lb_\al \lb_\ga)-\lb_\al
 \mathcal{ A}_{_{\lb}}(\lb_\ga)-\lb_\ga\mathcal{A}_{_{\lb}}(\lb_{\al})\\
&= 2 (\text{coefficient of $\frac{\partial^2}{\partial \lb_\al \lb_{\beta}}$ in $\mathcal{ A}_{_{\lb}}$})\\
&=\frac{2}{\beta}\dt_{\al\ga}
 \end{align*}
and similarly,
\begin{equation}
\frac{d\mu_id\mu_j}{dt}=\frac{2}{\beta}\dt_{ij}.
\label{diag}
\end{equation}
These identities can also be computed from the expressions (\ref{sde}) of $d\lb_\al$ in terms of the $\lb_i,\mu_j$, as done in the remark below. 
From Ito's formula~\ref{3.5}, it then follows that 
\begin{multline*}
(d\lb_1,\ldots,d\lb_n,d\mu_1,\ldots,d\mu_{n-1})
\\
=(\mathcal{A}_\DOU  \lb_1,\ldots,\mathcal{ A}_\DOU \lb_n, 
\mathcal{A}_\DOU \mu_1,\dots,\mathcal{A}_\DOU \mu_{n-1})dt
+\sigma (\lb,\mu)db_t,
\end{multline*}
where, according to (\ref{restr}) and (\ref{3.13''}), 
\begin{equation}
\label{3.23''}
\begin{split}
\mathcal{A}_\DOU (\lb_i)&=
\mathcal{A}_\lambda (\lb_i)=-\lb_i  +\sum_{j\neq i}\frac{1}{\lb_i-\lb_j},\\
\mathcal{A}_\DOU (\mu_i)&=\mathcal{A}_\mu (\mu_i)= -\mu_i  +\sum_{j\neq i}
\frac{1}{\mu_i-\mu_j},
\end{split}
\end{equation}
establishing the form of $F_\al^{(n)}(\lb)$ in (\ref{sde}), thus yielding (\ref{SDE-lambda0}).
Identities~(\ref{SDE-lambda0}), (\ref{3.23}) and~(\ref{diag}), 
together with Ito's formula~\eqref{3.5}, then establish the formula (\ref{1.12''}) 
for $\AR_{\lb\mu}$\footnote{As an alternative way, (\ref{restr}) 
and~(\ref{3.23}) suffice to establish~(\ref{1.12''}), with~(\ref{3.23}) 
needed to establish the coupling~$\AR_{\lb\mu}$.}.
\end{proof}

\begin{remark} Note that the identities (\ref{diag}) can be computed as well from the SDE (\ref{sde}), using residue calculations:
\begin{equation}
\label{3.23'}
\begin{split}
\frac{d\lb_\al d\lb_\ga}{dt}&=\frac{2}{\beta}\left(\frac{P_{n-1}(\lb_\al)}{P_n'(\lb_\al)}\right) 
\left(\frac{P_{n-1}(\lb_\ga)}{P_n'(\lb_\ga)}\right) \\
&\quad \left(
\sum_{1\leq i<j\leq n-1}\frac{ {2}~r_i^2r_j^2~}{(\lb_{\al}\!-\!\mu_i) (\lb_{\al}\!-\!\mu_j)(\lb_{\ga}\!-\!\mu_i) (\lb_{\ga}\!-\!\mu_j)}%
\right.
\\
&\qquad
\left.+\sum^{n-1}_{i=1}\frac{r^4_i }{(\lb_{\al}\!-\!\mu_i)^2(\lb_{\ga}\!-\!\mu_i)^2} 
+\sum_{i=1}^{n-1}\frac{2~r_i^2  }{(\lb_{\al}\!-\!\mu_i)(\lb_{\ga}\!-\!\mu_i)}+1\right) = \frac{2}{\beta}\dt_{\al\gamma}\\
\frac{d\mu_id\mu_j}{dt}&=\frac{2}{\beta}\dt_{ij}.
\end{split}
\end{equation}
\end{remark}

\bigbreak


\noindent We now turn to the proof of Corollaries~\ref{general} and \ref{Cor:unis}:

\begin{proof}[Proof of Corollary \ref{general} ]
Note, using logarithmic derivatives, that  
$$
\bigl(\mbox{Invariant measure (\ref{1.29})}\bigr) \beta {\mathcal{A}}^{\top}
\bigl(\mbox{Invariant measure (\ref{1.29})}\bigr)^{-1}
$$
is a quadratic polynomial in $\beta$, which by Theorem \ref{Th:Dyson2} vanishes for $\beta=1,2,4$ and thus it vanishes identically in $\beta$. That the process restricted to $\lambda$ or $\mu$ is the standard Dyson process follows from the form of the generator ${\mathcal{A}}^{\top}$. 

\end{proof}

\begin{proof}[Proof of Corollary~\ref{Cor:unis}.] 
In order to study the stochastic behavior of $\mu_\alpha-\lb_\alpha$ and $\lb_{\alpha +1}- \mu_\alpha$ when $\mu_\alpha$ gets close to $\lb_{\alpha }$ or $\lb_{\alpha +1}$, one rewrites the Brownian part of $d(\lb_{\al}-\mu_\al)$ as follows:
\begin{equation}
\label{SDE-lambda0'}
\begin{aligned}
\lefteqn{ \sqrt{\frac \beta 2} \mbox{Brownian part of}~d(\lb_{\al}-\mu_\al)}\\
=&
\frac{P_{n-1}(\lb_{\al})}{P'_n(\lb_{\al})
}   \left(\sum_{1\leq i<j\leq n-1}\!\frac{\sqrt{2}~r_ir_j~\tilde{d  b_{ij}}}{(\lb_{\al}\! - \!\mu_i)(\lb_{\al}\!-\!\mu_j)}+\sum^{n-1}_{{i=1}\atop{i\neq \alpha}}\frac{r^2_idb_{ii}}{(\lb_{\al}\!-\!\mu_i)^2}  
   \displaystyle{ +\sum_{i=1}^{n-1}\frac{\sqrt{2}~r_i \tilde{d  b_{in}}}{\lb_{\al}\!-\!\mu_i}+db_{nn}\!}
  \right)  
  \\
  &+\left(\frac{P_{n-1}(\lb_{\al})}{P'_n(\lb_{\al})}\frac{r_\al^2}{(\lb_\al-\mu_\al)^2}-1\right)db_{\al\al}
.\end{aligned}
\end{equation}
At first notice that for $\mu_\al \simeq \lb_\al$, one has, using the expression (\ref{r}) for $r_k^2$,
$$ 
P_{n-1}(\lb_\al)={\mathcal O}(\mu_\alpha-\lb_\al )
,~~~ r_\al=\Oh (\sqrt{\mu_\al-\lb_\al})
\mbox{~~~~~and~~~~~~}  r_i=\Oh(1)~~\mbox{for}~~i\neq \al, $$
from which one deduces that the first line on the right hand side of (\ref{SDE-lambda0'}) has order ${\mathcal O}(\sqrt{\mu_\alpha-\lb_\al })$. Using again (\ref{r}), the second line of (\ref{SDE-lambda0'})  multiplied with $P'_n(\lb_{\al})P'_{n-1}(\mu_{\al}) $  reads:
$$\begin{aligned}
P'_{n-1}(\mu_{\al}) P'_n(\lb_{\al})
 &\left(\frac{P_{n-1}(\lb_{\al})}{P'_n(\lb_{\al})}\frac{r_\al^2}{(\lb_\al-\mu_\al)^2}-1\right)
\\ \\
&=\left(\frac{P_{n-1}(\lb_\al)-P_{n-1}(\mu_\al)}{\lb_\al-\mu_\al}\right)
\left(\frac{P_{n }(\mu_\al)-P_{n }(\lb_\al)}{\mu_\al-\lb_\al}\right)- P'_{n-1}(\mu_{\al}) P'_n(\lb_{\al})
\\
&=\Oh (\mu_\al-\lb_\al),\end{aligned}
$$
Then 
$$
\begin{aligned}
\lefteqn{\frac{dt \mbox{-part of }~d(\mu_{\al}-\lb_\al)}{dt}\Bigr|_{\mu_\al=\lb_\al}}
\\&=
\Bigl(\lb_\al-\mu_\al+\sum_{{1\leq j\leq n-1}\atop{j\neq \al}}
\frac 1{\mu_\al-\mu_j}-
\sum_{{1\leq j\leq n}\atop{j\neq \al}}
\frac 1{\lb_\al-\lb_j}\Bigr)\Bigr|_{\mu_\al=\lb_\al}
\\&=
\sum_{{1\leq j\leq n-1}\atop{j\neq \al}} 
\frac{\mu_j-\lb_j}{(\lb_\al-\mu_j)(\lb_\al-\lb_j)}
+\frac{1}{\lb_n-\lb_\al}>0,\end{aligned}
$$
which follows from the inequalities (for $1\leq j\leq n-1$), 
$$
\mu_j-\lb_j\geq 0 ,~~~
(\lb_\al-\mu_j)(\lb_\al-\lb_j)\geq 0
~~~\mbox{~~and~~}~~~ \lb_n-\lb_\al>0.
$$
This proves the first relation (\ref{unis}), while the second one is done in a similar way. 
\end{proof}

\section{The eigenvalues of three consecutive minors}
\label{sec:threeminors}

In this section we shall prove Theorem \ref{Theo:3minors}, which affirms that for the Dyson process the joint spectra of any three consecutive minors is not Markovian, although the Markovianess of the spectra holds for any one or any two consecutive minors.


  Note that given an It\^o diffusion $X_t \in  \BR^n$, with stochastic differential equation
$dX_t = a (X_t) dt + \sigma (X_t) db_t$, as in  (\ref{3.1}), and generator $\mathcal{A}$,
the process restricted to $Y_i = \varphi_i (X) , 1\leq i \leq \ell$ is not Markovian (at least for generic initial conditions) if the generator fails to preserve the field of functions $\mathcal{F} (Y)$ generated by the $(Y_1,\ldots,Y_\ell):=(\varphi_1 (X),\ldots,\varphi_\ell(X)) $, i.e. 
\begin{equation}
\mathcal{A} \mathcal{F} (Y) \nsubseteq \mathcal{F} (Y).
\label{4.1}
\end{equation}
and provided the diffusion does not hit the $Y$-boundary of the domain. 

\medskip

\begin{proof}[Proof of Theorem~\ref{Theo:3minors}.]  In order to show the non-Markovianess of
\[
\Gamma := (\lb,\mu,\nu) = 
(\spec B, \spec B^{(n-1)}, \spec B^{(n-2)})
\]
 it suffices to find a function, such that the function, obtained by applying the Dyson-generator to it, is not a function of $(\lb,\mu,\nu)$. We pick a function of the product form  $xy = g(\Gamma)h(\Gamma)$, where
  $$x:=g(\Gamma):=\sum^{n-2}_{i=1} B_{ii}~~\text{and}~~y:=h(\Gamma):= \det B$$ 
  are two independent functions. Then, according to formula (\ref{3.2})
 \begin{equation}
 \mathcal{A}_\DOU xy=
   \frac{dxdy}{dt}+x\mathcal{A}_\DOU y+
   y\mathcal{A}_\DOU x.
\label{6.2} \end{equation}
Since $x$ and $\mathcal{A}_\DOU x$ are functions of $\nu$ only and since $y$ and $\mathcal{A}_\DOU y$ are functions of $\lambda$ only, $x\mathcal{A}_\DOU y+
   y\mathcal{A}_\DOU x$ is a function of $(\lb, \nu)$ only. Therefore, to establish non-Markovianess of $(\lb,\mu,\nu)$, it suffices to show that $\frac{dxdy}{dt}$ is not only a function of $(\lb,\mu,\nu)$. 
   Since, by It\^o's formula (\ref{3.5}), 
\begin{align*}
dx\,dy &=
\sum^{n-2}_1 dB_{ii} 
\left(\sum^n_1  \frac{\partial\det  B}{\partial B_{jj}} dB_{jj} + 
\sum_{1\leq i < j \leq n } \sum_{\ell = 0} ^ { \beta-1}  
\frac{\partial  \det B}{\partial B^{(\ell)}_{ij}} dB^{(\ell)}_{ij}   \right)
\\
&=
\frac{2}{\beta} dt \sum^{n-2}_1    
\frac{\partial \det  B}{\partial B_{ii}} = 
\frac{2}{\beta}   \sum^{n-2}_{i=1} \det(\minor_{ii} (  B) ) dt ,
\end{align*}
it suffices to show that the right hand side is not a function of $(\lb,\mu,\nu)$ only.
Here $\minor_{ii}$ denotes removing row $i$ and column $i$ of the matrix.


 
For example in the case $\beta=2, n=3$, this amounts to showing that the determinant of the lower-right $2\times 2$ principal minor of $B$ is not a function of $(\lb,\mu,\nu)$ only; to do this, it is convenient to reparametrize the matrix as
$$
B=\begin{pmatrix}
B_{11}&\rho_3e^{i\eta_3}&\rho_2e^{-i\eta_2}\\
\\
\rho_3e^{-i\eta_3}&B_{22}&\rho_1e^{i\eta_1}\\
\\
\rho_2e^{i\eta_2}&\rho_1e^{-i\eta_1}&B_{33}\end{pmatrix}
.$$
Using the following formulae
\begin{align*}
B_{11}&=\nu_1,\\
 B_{22}&=\mu_1+\mu_2-\nu_1,\\
 B_{33}&=\lb_1+\lb_2+\lb_3-\mu_1-\mu_2  \\
\rho_3^2&=(\mu_2-\nu_1)(\nu_1-\mu_1),
\end{align*}
the lower-right $2\times 2$ principal minor of $B$ reads
\begin{align*}
\det(\minor_{11} (  B) )&=B_{22} B_{33} - \rho_1^2  \\
&=(\mu_1+\mu_2-\nu_1)(\lb_1+\lb_2+\lb_3-\mu_1-\mu_2)-\rho_1^2.
\end{align*}
%
One observes that
\begin{align*}
0&=\det B-\lb_1\lb_2\lb_3\\
&=B_{11}B_{22}B_{33}-\lb_1\lb_2\lb_3-\sum_{i=1}^3\rho^2_i B_{ii}+2\rho_1\rho_2\rho_3\cos(\eta_1+\eta_2+\eta_3)\\
&=F_1(\lb,\mu,\nu)-\rho^2_1\nu_1-\rho^2_2(\mu_1+\mu_2-\nu_1)\\
&\qquad +~2\rho_1\rho_2\sqrt{(\mu_2-\nu_1)(\nu_1-\mu_1)}~\cos(\eta_1+\eta_2+\eta_3)
\end{align*}
and
\begin{multline*}
0=\Tr B^2-(\lb^2_1+\lb^2_2+\lb^2_3)=\sum^3_{i=1} B_{ii}^2+2\sum_{i=1}^3\rho^2_i-(\lb^2_1+\lb_2^2+\lb^2_3)\\
=F_2(\lb,\mu,\nu)+2(\rho^2_1+\rho^2_2),
\end{multline*}
where $F_i(\lb,\mu,\nu)$ are functions of the spectral data $(\lb,\mu,\nu)$. 
Upon solving these two equations in $\rho_1$ and $\rho_2$, one notices that, 
in particular, $\rho_1$ is a function of $\cos(\eta_1+\eta_2+\eta_3)$ and the 
spectral data $(\lb,\mu,\nu)$, hence showing that 
$\det(\minor_{11} (  B) )$ is not a function of $(\lb,\mu,\nu)$ only; thus the 
same is true for $\mathcal{A}_{\text{\tiny Dys}}xy$. This proves that 
$\mathcal{A}_{\DOU}xy$ does not belong to the field of functions depending 
on $(\lb,\mu,\nu)$.

More generally, by a perturbation argument about 
 $  B^{(n-1)}= \diag (\mu_1,\ldots,\mu_{n-1}),$ one shows similarly that
 $$\sum^{n-2}_{i=1} \det(\minor_{ii} (  B) ) 
 \not\in \mathcal{F} (\lb,\mu,\nu),$$ 
for $\beta=2$ and $4$.    
      
   Finally, the boundary of the process $(\lb, \mu,\nu)$ 
is given by the subvariety where some of the $\mu_i$'s hit the $\lb_j$'s or the $\nu_k$'s; that is when $P_n(\mu_i)=0$ or $P_{n-1}(\nu_k)=0$ for some $1\leq i\leq n-1$ or for some $1\leq k\leq n-2$; 
   $r_j^2(\mu,\nu)=0$ for $1\leq j\leq n-2$. From Corollary \ref{Cor:unis}, on sees that the process never reaches that boundary. 
   This ends the proof of Theorem \ref{Theo:3minors}. 
\end{proof}

\bibliography{GUE_Minor_process2}{}
\bibliographystyle{alpha}
\end{document}